\documentclass[11pt,a4paper,reqno]{amsart}

 \usepackage{amsfonts,amsmath,amscd,amssymb,amsbsy,amsthm,amstext,amsopn}
 \usepackage{fullpage,mathrsfs,subfigure}
 \usepackage{pstricks,pst-node,pst-text,pst-tree,multicol}
   \usepackage{epic}
 \usepackage{stmaryrd}  
 \usepackage{extarrows}
\usepackage[all]{xy}

  \usepackage{float}

 \numberwithin{equation}{section}
\newtheorem{theorem}{Theorem}[section]
\newtheorem{proposition}[theorem]{Proposition}
\newtheorem{lemma}[theorem]{Lemma}

\theoremstyle{definition}
\newtheorem{definition}[theorem]{Definition}
\newtheorem{example}[theorem]{Example}

\theoremstyle{remark}
\newtheorem{remark}[theorem]{Remark}

\def\gbf#1{\mbox{\boldmath$#1$}} 
 
\newcommand{\kk}{\ensuremath{\Bbbk}} 
\newcommand{\CC}{\ensuremath{\mathbb{C}}}
\newcommand{\NN}{\ensuremath{\mathbb{N}}}

\newcommand{\RR}{\ensuremath{\mathbb{R}}} 
\newcommand{\ZZ}{\ensuremath{\mathbb{Z}}} 
\newcommand{\one}{\ensuremath{(\mathrm{i})}}
\newcommand{\two}{\ensuremath{(\mathrm{ii})}}
\newcommand{\three}{\ensuremath{(\mathrm{iii})}}
\newcommand{\four}{\ensuremath{(\mathrm{iv})}}

\renewcommand{\div}{\operatorname{div}} 

\newcommand{\ghilb}{\ensuremath{G}\operatorname{-Hilb}}

\newcommand{\lcm}{\operatorname{lcm}}

\newcommand{\sign}{\operatorname{sign}}

\newcommand{\supp}{\operatorname{supp}}

 \newcommand{\Cl}{\operatorname{Cl}}

\newcommand{\Hom}{\operatorname{Hom}} 
\newcommand{\im}{\operatorname{im}}

\newcommand{\Pic}{\operatorname{Cl}}

\newcommand{\SL}{\operatorname{SL}}

\newcommand{\syz}{\operatorname{syz}}

\title{Cohomology of wheels on toric varieties}

 \author{Alastair Craw} 
 \address{Department of Mathematical Sciences\\ University of Bath,\\ Bath BA2 7AY\\ UK}
 \email{A.Craw@bath.ac.uk}
  \author{Alexander Quintero V\'{e}lez}
 \address{Instituto de Matem\'{a}ticas, Universidad de Antioquia, Calle 67 No. 53--108, Medell\'{i}n, Colombia}
 \email{A.QuinteroVelez@ciencias.udea.edu.co} 

\begin{document}
\bibliographystyle{plain}

 \begin{abstract}
 We describe explicitly the cohomology of the total complex of certain diagrams of invertible sheaves on normal toric varieties. These diagrams, called wheels, arise in the study of toric singularities associated to dimer models. Our main tool describes the generators in a family of syzygy modules associated to the wheel in terms of walks in a family of graphs. 
 \end{abstract}

 \maketitle

 \section{Introduction}
A standard tool in homological algebra is to study a finitely generated module over a ring in terms of a free resolution, or more generally, a coherent sheaf on a variety in terms of a resolution by locally free sheaves. Conversely, given a complex $T^\bullet$ of locally free sheaves on a variety $X$, it is natural to ask whether the cohomology of the complex is nonzero in one degree only, say $k\in \ZZ$, in which case $T^\bullet$ is quasi-isomorphic to the pure sheaf $H^k(T^\bullet)[-k]$.  In particular, it is important to have an explicit understanding of the cohomology sheaves of a complex of locally free sheaves. Our main result achieves this for a class of four-term complexes of locally free sheaves on normal toric varieties. 

Our motivation comes from the study of derived categories of toric varieties associated to consistent dimer model algebras (see Bocklandt--Craw--Quintero-V\'{e}lez~\cite[Section~2.4]{BCQ12} for a brief introduction). The best-known example of a consistent dimer model algebra is the skew group algebra $\CC[x,y,z]*G$ for a finite abelian subgroup $G\subset \SL(3,\CC)$, in which case the relevant toric variety is the $G$-Hilbert scheme $X=\ghilb(\CC^3)$ introduced by Nakamura~\cite{Nakamura01}. In their study of the equivalence of derived categories induced by the universal family on the $G$-Hilbert scheme, Cautis--Logvinenko~\cite{CL09} describes explicitly the cohomology sheaves of certain four-term complexes $T^\bullet$ on $X$ and hence shows that with only one exception, every such complex is quasi-isomorphic to a pure sheaf $H^k(T^\bullet)[-k]$ for $k=0,1$ (see also Cautis--Craw--Logvinenko~\cite{CCL12}). Our main result (see Theorem~\ref{thm:mainintro} below) can be applied to a broader class of four-term complexes, including those arising in the study of the derived equivalences induced by the universal family of fine moduli spaces $X$ associated to any consistent dimer model algebra. As an application, joint work with Raf Bocklandt~\cite{BCQ12} establishes the dimer model analogue of the Cautis-Logvinenko result, namely, that for a special choice of moduli space generalising the $G$-Hilbert scheme, all but one of the four-term complexes $T^\bullet$ on $X$ obtained from the derived equivalence is quasi-isomorphic to a pure sheaf $H^k(T^\bullet)[-k]$ for $k=0,1$.

The complexes $T^\bullet$ that we consider in this paper are four-term complexes of the form 
\begin{equation}
\label{eqn:Tbullet}
L \xlongrightarrow{d^3}  \bigoplus_{j=1}^m L_{j,{j+1}} \xlongrightarrow{d^2} \bigoplus_{j=1}^m L_j \xlongrightarrow{d^1} L
\end{equation}
for some $m\geq 2$, where $L$, $L_{j,j+1}$ and $L_{j}$ ($1\leq j\leq m$) are invertible sheaves on any normal toric variety $X$, where each differential is equivariant with respect to the torus-action on $X$, and where the right-hand copy of $L$ lies in degree zero. Assume in addition that for $1\leq j\leq m$, the restriction of the differential $d^2$ to the summand $L_{j,j+1}$ has image in $L_j\oplus L_{j+1}$ (with indices modulo $m$). This means that if we separate vertically the summands in the terms of $T^\bullet$ and hence break the matrices defining the differentials into their constituent maps between summands, the complex can be presented as a diagram of the form 
\begin{equation}
\label{eqn:diagram}
\begin{split}
    \centering    
         \psset{unit=0.45cm}
     \begin{pspicture}(0,-1)(25,13.7)
\cnodeput*(0,6){A}{$L$} 
\cnodeput*(8,12){B}{$L_{1,2}$}
\cnodeput*(8,9){C}{$L_{2,3}$} 
\cnodeput*(8,6){D}{$L_{3,4}$}
\cnodeput*(8,3.2){S}{$\vdots$}
\cnodeput*(8,0){E}{$L_{m,1}$}
\cnodeput*(18,12){F}{$L_{1}$}
\cnodeput*(18,9){G}{$L_{2}$}
\cnodeput*(18,6){H}{$L_{3}$}
\cnodeput*(18,3.2){T}{$\vdots$}
\cnodeput*(18,0){I}{$L_{m}$}
\cnodeput*(26,6){J}{$L.$}
\psset{nodesep=1pt}
   \ncline{->}{A}{B}\lput*{:U}(0.6){$\scriptstyle{D_{1,2}}$}
   \ncline{->}{A}{C}\lput*{:U}(0.6){$\scriptstyle{D_{2,3}}$}
   \ncline{->}{A}{D}\lput*{:U}(0.6){$\scriptstyle{D_{3,4}}$}
   \ncline{->}{A}{E}\lput*{:U}(0.6){$\scriptstyle{D_{m,1}}$}
 \ncline{->}{B}{F}\lput*{:U}(0.4){$\scriptstyle{D^2_1}$}
  \ncline{->}{B}{G}\lput*{:U}(0.4){$\scriptstyle{D^1_2}$}
  \ncline{->}{C}{G}\lput*{:U}(0.4){$\scriptstyle{D^3_2}$}
  \ncline{->}{C}{H}\lput*{:U}(0.4){$\scriptstyle{D^2_3}$}
\ncline{->}{D}{H}\lput*{:U}(0.4){$\scriptstyle{D^4_3}$}
\ncline{->}{E}{I}\lput*{:U}(0.4){$\scriptstyle{D^1_m}$}
\nccurve[angleA=-40,angleB=140]{->}{E}{F}\lput*{:U}(0.4){$\scriptscriptstyle{D^m_1}$}
 \ncline{->}{F}{J}\lput*{:U}(0.4){$\scriptstyle{D^{1}}$}
   \ncline{->}{G}{J}\lput*{:U}(0.4){$\scriptstyle{D^{2}}$}
   \ncline{->}{H}{J}\lput*{:U}(0.4){$\scriptstyle{D^{3}}$}
   \ncline{->}{I}{J}\lput*{:U}(0.4){$\scriptstyle{D^{m}}$}
       \end{pspicture}
 \end{split}
 \end{equation}
 The maps between invertible sheaves in this diagram are multiplication by a torus-invariant section of an invertible sheaf on $X$. We illustrate this and fix notation by writing on each arrow in diagram \eqref{eqn:diagram} the Cartier divisor of zeros of the corresponding section so, for example, the effective divisor $D^1_{2}\in H^0(L_2\otimes L_{1,2}^{-1})\cong \Hom(L_{1,2},L_2)$ denotes the Cartier divisor of zeros of the section that defines the map from $L_{1,2}$ to $L_2$. This diagram can be represented equally well in a planar picture that is reminiscent of a bicycle wheel (see Figure~\ref{fig:wheel} in Section~\ref{sec:cohomologyWheels}), and we refer to any such four-term complex $T^\bullet$ as a `wheel' on $X$. 

 To state our main result we choose once and for all a rather special order on the set of transpositions of $m$ letters (see Section~\ref{sec:syzygies}), giving $\tau_1= (\mu_1,\nu_1),\dots, \tau_n=(\mu_n,\nu_n)$ where $n= \binom{m}{2}$ and $\mu_k<\nu_k$ for $1\leq k\leq n$. In addition, for every index $1\leq k\leq n$ we define a subscheme $Z_k\subset X$ to be the scheme-theoretic intersection of certain torus-invariant divisors in $X$. To be more precise, let $\mathscr{D}:=\{D_{\lambda}\}_{\lambda\in \Lambda}$ be a set of torus-invariant divisors in $X$. Define the greatest common divisor and the least common multiple of the set $\mathscr{D}$ to be the torus-invariant divisors  
\[
\gcd(\mathscr{D}) = \max \{ D \mid D_\lambda-D\geq 0 \;\forall \;\lambda\in \Lambda\}\quad \text{and}\quad \lcm(\mathscr{D}) = \min \{ D \mid D-D_\lambda\geq 0 \;\forall \;\lambda\in \Lambda\}
\]
 respectively; here max/min means choose the maximal/minimal values for the coefficients of each prime divisor in the expression for $D$. Define subschemes $Z_k \subset X$ for $1 \leq k \leq n$ in terms of the Cartier divisors labelling the arrows in diagram \eqref{eqn:diagram} as follows:
\begin{enumerate}
\item[\one] for $1\leq k\leq m$, define $Z_k$ to be the scheme-theoretic intersection of $\gcd(D_{k+1}^k,D^{k+1}_k)$ and the divisor $\lcm\big(D^1,\dots,D^m,\gcd(D_{k+2}^{k+1},D^{k+2}_{k+1}),\dots,\gcd(D_{1}^{m},D^{1}_{m})\big)-\lcm(D^k,D^{k+1})$;
\item[\two] for $m+1 \leq k \leq 2m-3$, define $Z_k$ to be the scheme-theoretic intersection of the divisors $\lcm(D^1,D^{\nu_k},D^{\nu_k+1},\dots,D^m)-\lcm(D^{1},D^{\nu_k})$ and $\lcm(D^1,D^{\nu_k-1},D^{\nu_k})-\lcm(D^{1},D^{\nu_k})$;
\item[\three] for $2m-2\leq k \leq n$, define $Z_k$ to be the scheme-theoretic intersection of the divisors $\lcm(D^{\mu},D^{\mu_k},D^{\nu_k})-\lcm(D^{\mu_k},D^{\nu_k})$ for $\mu \in \{1, \dots, \mu_k-1\}\cup\{\nu_k-1\}$.
\end{enumerate}
 The subschemes $Z_k\subset X$ are torus-invariant, though some (possibly all) may be empty, see Example~\ref{exa:hex} for an explicit calculation. 
 
\begin{theorem}
\label{thm:mainintro}
Let $X$ be a normal toric variety and let $T^\bullet$ be the complex from \eqref{eqn:Tbullet}, with differentials determined by the Cartier divisors shown in \eqref{eqn:diagram}. Then:
\begin{enumerate}
\item[(1)]$H^0(T^{\bullet}) \cong \mathscr{O}_Z \otimes L$ where $Z$ is the scheme-theoretic intersection of $D^1,\dots,D^m;$
\item[(2)]$H^{-1}(T^{\bullet})$ has an $n$-step filtration 
\[
\im(d^2)=F^0 \subseteq F^1\subseteq \cdots \subseteq F^{n-1}\subseteq F^n=\ker(d^1)
\]
 where, for $1\leq k \leq n$ and for the permutation $\tau_k=(\mu_k,\nu_k)$, we have
\begin{equation}
\label{eqn:sheafquotient1intro}
F^k/F^{k-1}\cong \mathscr{O}_{Z_k} \otimes L_{\mu_k}\otimes
L_{\nu_k}\otimes L^{-1}(\gcd(D^{\mu_k},D^{\nu_k}));
\end{equation}
\item[(3)]$H^{-2}(T^{\bullet}) \cong \mathscr{O}_D \otimes L(D)$ where $D=\gcd(D_{1,2},D_{2,3},\dots,D_{m,1});$
\item[(4)]$H^{-3}(T^{\bullet})\cong 0$.
\end{enumerate}
\end{theorem}

To prove Theorem~\ref{thm:mainintro} we lift the complex $T^\bullet$ to a complex of $\Cl(X)$-graded $S$-modules using the functor of Cox~\cite{Cox95}, where $\Cl(X)$ and $S$ denote the class group and Cox ring of $X$ respectively. Explicitly, if $S(L)$ denotes the free $S$-module with generator in degree $L\in \Cl(X)$, then $T^\bullet$ can be lifted to the complex 
\begin{equation}
\label{eqn:complexSmods1}
S(L) \xlongrightarrow{\varphi^3}  \bigoplus_{j=1}^m S(L_{j,{j+1}}) \xlongrightarrow{\varphi^2} \bigoplus_{j=1}^m S(L_j) \xlongrightarrow{\varphi^{1}} S(L).
\end{equation}
 This translates the problem to one from commutative algebra. The lion's share of the effort in proving Theorem~\ref{thm:mainintro} goes into proving part (2). For this, the image of $\varphi^2$ is generated by elements $\gbf{\alpha}_1,\dots, \gbf{\alpha}_m$, and our chosen order on the set of transpositions on $m$ letters determines an order on the generators $\gbf{\beta}_1,\dots, \gbf{\beta}_n$ of $\ker(\varphi^1)$ which in turn defines a filtration
\[
\im(\varphi^2) = F^0 \subseteq F^1\subseteq F^2\subseteq \cdots \subseteq F^{n-1}\subseteq F^n=\ker(\varphi^1).
\]
We give a presentation for each successive quotient $F^k/F^{k-1}$ as a cyclic $\Cl(X)$-graded $S$-module of the form $(S/I_k)(L_{\mu_k}\otimes L_{\nu_k}\otimes L^{-1}(\gcd(D^{\mu_k},D^{\nu_k})))$ for some monomial ideal $I_k$ whose generators are defined via the Cartier divisors $D^1,\dots, D^m$ labelling the right-hand arrows in the diagram \eqref{eqn:diagram} illustrating the wheel (see Proposition~\ref{prop:filtration2}). This calculation can be performed in any given example using Macaulay2~\cite{M2}, but we present a unified description for all $1\leq k\leq n$.  (Warning: M2 may choose an order on the generators $\gbf{\beta}_1,\dots, \gbf{\beta}_n$ that differs from ours, see Remark~\ref{rem:hex}.)

Our main tool, which may be of independent interest, is a description of the syzygy module of $\ker(\varphi^1)$ in terms of walks in the complete graph $\Gamma$ on $m$ vertices. In fact, for each $1\leq k\leq n$ we introduce a subgraph $\Gamma_k$ of $\Gamma$ that enables us to describe uniformly the module of syzygies $\syz(F^k)$ in terms of certain walks in $\Gamma_k$. To state the result, recall that a circuit in $\Gamma_k$ is a closed walk that does not pass through a given vertex twice. It is straightforward to associate a syzygy to every such circuit (see Lemma~\ref{lem:syzygy}). A circuit is said to be minimal if it admits no chords (see \eqref{eqn:splitting}). We prove the following result (see Theorem~\ref{thm:main1}).

\begin{theorem}
\label{thm:intro}
For $m\leq k\leq n$, the module $\syz(F^{k})$ is generated by the set of syzygies associated to the minimal circuits of $\Gamma_k$.
\end{theorem}

\noindent The precise description of the syzygies from Theorem~\ref{thm:intro} allows us to read off directly a set of monomial generators for each ideal $I_k$, and this feeds into the proof of Theorem~\ref{thm:mainintro} above. Generating sets for toric ideals arising from graphs were studied by Hibi--Ohsugi~\cite{OhsugiHibi}, and some of the graph-theoretic tools that we use here were also employed there. Properties of $\kk$-algebras arising from graphs have also been studied widely by Villarreal, see for example \cite{Villarreal}.  

Our main result was motivated by the statement of Cautis--Logvinenko~\cite[Lemma~3.1]{CL09} which asserts that  in the special case $m=3$, a version of Theorem~\ref{thm:mainintro} holds for the complex $T^\bullet$ from \eqref{eqn:Tbullet} arising from a diagram \eqref{eqn:diagram} on an arbitrary smooth separated scheme.  However, this is not true in general: the assertion \cite[Proof of Lemma 3.1(2)]{CL09} that certain elements $\beta_1, \beta_2, \beta_3$ generate $\ker(d^1)$ may fail if the maps from diagram \eqref{eqn:diagram} are not monomial maps. 

\begin{example}
\label{ex:counterexample}
For a counterexample in the notation of \emph{loc.cit.} (we write the signs explicitly), suppose the maps $L_1\to L$, $L_2\to L$ and $L_3\to L$ from \eqref{eqn:diagram} are defined locally near a point $p\in X$ as multiplication by $f_1:=x, f_2:=x+y, f_3:=y\in \mathscr{O}_{X,p}$. Then $(1,-1,1)$ lies in $\ker(d^1)$, but it does not lie in the submodule generated by $\beta_1 = (f_2,-f_1,0), \beta_2 = (-f_3,0,f_1), \beta_3= (0,f_3,-f_2)$. 
\end{example}

The assumption in Theorem~\ref{thm:mainintro} that $X$ is toric and the maps from \eqref{eqn:diagram} are torus-equivariant ensures that each map arises from multiplication by a monomial in the Cox ring of $X$, in which case standard Gr\"{o}bner theory shows that analogous elements $\gbf{\beta}_1, \gbf{\beta}_2, \gbf{\beta}_3$ generate the appropriate kernel (see Lemma~\ref{lem:1.1}). Under these additional assumptions, Remark~\ref{rem:CautisLogvinenko} explains how the statement of Cautis--Logvinenko~\cite[Lemma~3.1]{CL09} can be recovered as a special case of Theorem~\ref{thm:mainintro} when $X$ is smooth. The main results of both Cautis--Logvinenko~\cite{CL09} and Cautis--Craw--Logvinenko~\cite{CCL12} require the statement of \cite[Lemma~3.1]{CL09} only when $X$ is a smooth toric variety and the maps from \eqref{eqn:diagram} are torus-equivariant, so Theorem~\ref{thm:mainintro} holds at the level of generality required for both of those papers.

In fact, Theorem~\ref{thm:mainintro} provides a unified description of the sheaves \eqref{eqn:sheafquotient1intro} in the filtration on $H^{-1}(T^{\bullet})$ even for $m=3$, improving slightly on the statement from \cite[Lemma~3.1]{CL09}.  More generally, for $m>3$, the schemes $Z_k$ ($1\leq k\leq n$) divide naturally into three families determined by the intervals \one\ $1\leq k\leq m$; \two\ $m+1 \leq k \leq 2m-3$; and \three\ $2m-2\leq k \leq n$, leading to a more involved filtration in this case. That the statement is considerably more complicated for $m>3$ stems from the simple fact that any pair of vertices of a triangle are adjacent, while the same statement is not true for a polygon with $m>3$ vertices.

\medskip

\noindent\textbf{Acknowledgements.}  
Thanks to Raf Bocklandt for generating Example~\ref{exa:hex} and to Sonja Petrovic for comments on an earlier version of this paper. Thanks also to the anonymous referees for many helpful remarks. Our results owes much to experiments made with Macaulay2~\cite{M2}. Both authors were supported by EPSRC grant EP/G004048/1.

\section{Syzygies from walks in a complete graph}
\label{sec:syzygies}
Let $S=\kk[x_1,\dots, x_d]$ be a polynomial ring over a field $\kk$ and let $f^1, \dots, f^m\in S$ be monomials for some $m\geq 2$. Consider the free $S$-module with basis $\textbf{e}_1,\dots, \textbf{e}_m$ and define an $S$-module homomorphism $\varphi\colon \bigoplus_{\mu=1}^m S\mathbf{e}_\mu \longrightarrow S$ by setting $\varphi(\mathbf{e}_{\mu})=f^{\mu}$ for $1\leq \mu\leq m$.  For every pair of indices $1 \leq \mu < \nu \leq m$ we define monomials $f^{\mu,\nu}=\lcm(f^{\mu},f^{\nu})$ and set 
\begin{equation}
\label{eqn:Beta}
\gbf{\beta}_{(\mu,\nu)}= \frac{f^{\mu,\nu}}{f^{\nu}}\mathbf{e}_{\nu}-\frac{f^{\mu,\nu}}{f^{\mu}}\mathbf{e}_{\mu}.  
\end{equation}
The module of syzygies of $M:=\langle f^1,\dots, f^m\rangle$ is defined to be the $S$-module $\syz(M):= \ker(\varphi)$.  The following result is well known; see for example Eisenbud~\cite[Lemma~15.1]{Eisenbud95}.

\begin{lemma}
\label{lem:1.1}
The kernel of $\varphi$ is generated by the elements $\gbf{\beta}_{(\mu,\nu)}$ for $1 \leq \mu < \nu \leq m$.
\end{lemma}

It is convenient to order the set $\{ (\mu,\nu) \mid 1 \leq \mu < \nu \leq m \}$ of transpositions of $m$ letters. First list the transpositions of adjacent letters $\tau_j=(j,j+1)$ for $1 \leq j \leq m-1$. Set $\tau_m=(1,m)$, then list all remaining transpositions that involve $1$ as $\tau_j=(1,j-m+2)$ for $m+1 \leq j \leq 2m-3$, and finally list all remaining transpositions lexicographically, so $\tau_i=(\mu_i,\nu_i)$ precedes $\tau_j=(\mu_j,\nu_j)$ if and only if $\mu_i < \mu_j$ or $\mu_i=\mu_j$ and $\nu_i < \nu_j$. We may therefore list the generators of $\ker(\varphi)$ from Lemma~\ref{lem:1.1} by setting $\gbf{\beta}_j :=\gbf{\beta}_{(\mu_j,\nu_j)}$ for all $1 \leq j \leq n$, where $n= \binom{m}{2}$. This choice of order enables us to define for each $1 \leq k \leq n$ an $S$-module
 $$
 F^k=\langle \gbf{\beta}_1,\dots, \gbf{\beta}_{k}  \rangle.
 $$

 Our primary goal is to provide for each $1\leq k\leq n$ an explicit set of generators for the module of syzygies $\syz(F^k)$ that encodes the relations between $\gbf{\beta}_1,\dots, \gbf{\beta}_{k}$. Recall that this module is defined to be the kernel of the surjective $S$-module homomorphism
\[
\psi\colon \bigoplus_{j=1}^k S\gbf{\varepsilon}_j \longrightarrow F^k
\]
satisfying $\psi(\gbf{\varepsilon}_{j})= \gbf{\beta}_j$ for $1\leq j\leq k$. We compute this module directly for $1\leq k\leq m$.
 
 \begin{lemma}
 \label{lem:cyclic}
 The $S$-module $\syz(F^k)$ is the zero module for $1\leq k \leq m-1$, and it is a free module of rank one for $k=m$.
 \end{lemma}
 \begin{proof}
Our choice of order on transpositions ensures that for $1\leq k\leq m-1$,  there can be no relations between $\gbf{\beta}_1,\dots, \gbf{\beta}_{k}$. For $k=m$, let $\gbf{\sigma} = \sum_{j=1}^m s_j \gbf{\varepsilon}_{j}$ be a syzygy on $\gbf{\beta}_1,\dots, \gbf{\beta}_{m}$ where $s_1,\dots,s_m\in S$. By comparing coefficients of each $\mathbf{e}_{i}$ in the expression
\[
0 = \psi(\gbf{\sigma})=  s_m \bigg( \frac{f^{1,m}}{f^{m}}\mathbf{e}_{m}-\frac{f^{1,m}}{f^{1}}\mathbf{e}_{1} \bigg) + 
  \sum_{j=1}^{m-1} s_j \bigg( \frac{f^{j,j+1}}{f^{j+1}}\mathbf{e}_{j+1}-\frac{f^{j,j+1}}{f^{j}}\mathbf{e}_{j} \bigg) 
\]
we obtain the following equations
\begin{equation}
\label{eqn:comparecoeffs}
s_1f^{1,2} = s_2f^{2,3} = \cdots = s_{m-1}f^{m-1,m} = - s_m f^{1,m}.
\end{equation}
It's easy to see (or see Lemma~\ref{lem:syzygy} below for a proof) that the element
\begin{equation}
\label{eqn:sigma0}
\gbf{\sigma}_0:= -\frac{\lcm(f^1,\dots,f^m)}{f^{1,m}} \gbf{\varepsilon}_m + \sum_{j=1}^{m-1} \frac{\lcm(f^1,\dots,f^m)}{f^{j,j+1}} \gbf{\varepsilon}_j
\end{equation}
is a syzygy. Moreover, equations \eqref{eqn:comparecoeffs} imply that 
\[
\gbf{\sigma} = \frac{s_1f^{1,2}}{\lcm(f^1,\dots,f^m)}\gbf{\sigma}_0,
\]
so $\syz(F^m)$ is the free $S$-module with basis $\gbf{\sigma}_0$.
 \end{proof}
 
 We study the module $\syz(F^k)$ for $m+1\leq k\leq n$ by studying walks in a graph.  Let $\Gamma$ be the complete graph on $m$ vertices, with vertex set $\{1,2,\dots,m\}$. Assign an orientation to each edge $e=(\mu,\nu)$ by directing it from $\mu$ to $\nu$ if $\mu < \nu$.  Regard every such edge as being labelled by the corresponding generator $\gbf{\beta}_{(\mu,\nu)}$ of $\ker(\varphi)$. The order on the generators $\gbf{\beta}_1,\dots, \gbf{\beta}_n$ introduced above determines an order on the set of edges $e_1,\dots, e_n$ of $\Gamma$. A {\it walk $\gamma$ of length} $\ell$ in $\Gamma$ is a walk in the undirected graph that traverses precisely $\ell$ edges. Every such walk is characterised by the sequence of vertices $\gamma=(\mu_1,\mu_2,\dots, \mu_{\ell+1})$ in $\Gamma$ that it touches. A walk $\gamma$ is \emph{closed} if $\mu_1=\mu_{\ell+1}$, and a {\it circuit} is a closed walk for which $\mu_1,\dots, \mu_\ell$ are distinct. Each circuit $\gamma$ defines uniquely a subgraph of $\Gamma$, and we let $\supp(\gamma)$ denote its set of edges.  Given a circuit $\gamma$  and an edge $e\in \supp(\gamma)$, set $\sign_\gamma(e) = +1$ if $\gamma$ traverses $e$ according to the orientation in $\Gamma$, and set $\sign_\gamma(e) = -1$ if $\gamma$ traverses $e$ against orientation.

   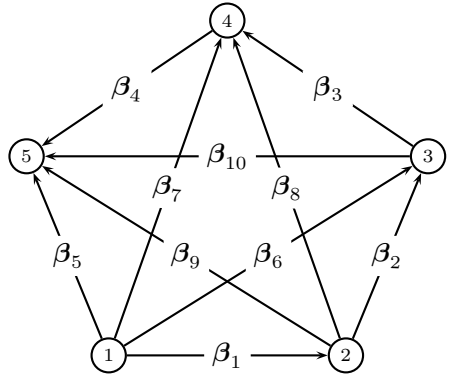
\begin{figure}[!ht]
 \centering
 \psset{unit=2.4cm}
 \begin{pspicture}(0,0)(2.2,1.9)
 \psset{linecolor=black}
 \cnodeput(0.45,0){A}{\tiny{$1$}} 
 \cnodeput(1.75,0){B}{\tiny{$2$}}
 \cnodeput(2.2,1.1){C}{\tiny{$3$}} 
  \cnodeput(1.1,1.85){D}{\tiny{$4$}}
 \cnodeput(0,1.1){E}{\tiny{$5$}}
 \ncline{->}{A}{B}\lput*{:0}{\small{$\gbf{\beta}_1$}}
 \ncline{->}{B}{C}\lput*{:290}{\small{$\gbf{\beta}_2$}}
 \ncline{->}{C}{D}\lput*{:220}{\small{$\gbf{\beta}_3$}}
 \ncline{->}{D}{E}\lput*{:142}{\small{$\gbf{\beta}_4$}}
  \ncline{<-}{E}{A}\lput*{:70}{\small{$\gbf{\beta}_5$}}
  \ncline{->}{A}{C}\lput*{:330}{\small{$\gbf{\beta}_6$}}
 \ncline{<-}{D}{A}\lput*{:105}{\small{$\gbf{\beta}_7$}}
 \ncline{<-}{D}{B}\lput*{:70}{\small{$\gbf{\beta}_8$}}
 \ncline{<-}{E}{B}\lput*{:30}{\small{$\gbf{\beta}_9$}}
 \ncline{<-}{E}{C}\lput*{:0}{\small{$\gbf{\beta}_{10}$}}
 \end{pspicture}
 \caption{Directed graph $\Gamma$ illustrating generators $\gbf{\beta}_1,\dots, \gbf{\beta}_n$ for $m=5$}
  \label{fig:quiver}
 \end{figure}
 
Given that the elements $\gbf{\beta}_j$ for $1\leq j\leq n$ correspond to edges in $\Gamma$, we may index the basis elements $\gbf{\varepsilon}_j$ for $1\leq j\leq n$ by edges $e_1,\dots, e_n$ in $\Gamma$. Thus, for the edge $e=e_j$ for $1\leq j\leq n$, we write $\gbf{\varepsilon}_e:=\gbf{\varepsilon}_j$. For any vertices $\mu_1, \dots, \mu_{\ell+1}$ in $\Gamma$, set 
 \[
 f^{\mu_1,\dots,\mu_{\ell+1}}=\lcm(f^{\mu_{1}}, \dots, f^{\mu_{\ell+1}}).
 \]
 For a walk $\gamma=(\mu_1,\mu_2,\dots, \mu_{\ell+1})$ in $\Gamma$ we define the monomial $f^\gamma:= f^{\mu_1,\dots,\mu_{\ell+1}}$. In particular, for an edge $e$ in $\Gamma$ joining vertex $\mu$ to $\nu$, we obtain $f^e=f^{\mu,\nu}$.

\begin{lemma}
\label{lem:syzygy}
For any circuit $\gamma$ of length at least three in $\Gamma$, the vector
$$
\gbf{\sigma}_{\gamma}= \sum_{e\in \supp(\gamma)} \sign_\gamma(e) \frac{f^{\gamma}}{f^e} \gbf{\varepsilon}_{e}
$$
is a syzygy on $\gbf{\beta}_1,\dots, \gbf{\beta}_{n}$.
\end{lemma}
\begin{proof}
 If $\gamma$ has length two then $\gbf{\sigma}_{\gamma}=\gbf{\varepsilon}_e-\gbf{\varepsilon}_e=0$ which is not in fact a syzygy by definition. For any circuit $\gamma$ of length at least three we must show that 
\[
\psi(\gbf{\sigma}_{\gamma})=\sum_{e\in \supp(\gamma)} \sign_\gamma(e) \frac{f^{\gamma}}{f^e}\gbf{\beta}_{e} = 0. 
\]
For an edge $e$ that $\gamma$ traverses in the direction from vertex $\mu$ to vertex $\mu^\prime$, we have that 
$$
\sign_\gamma(e) \frac{f^{\gamma}}{f^e}\gbf{\beta}_{e}= \frac{f^{\gamma}}{f^e}\bigg(\frac{f^{e}}{f^{\mu^\prime}}
\mathbf{e}_{\mu^\prime}-\frac{f^{e}}{f^{\mu}}\mathbf{e}_{\mu}\bigg) = \frac{f^{\gamma}}{f^{\mu^\prime}}
\mathbf{e}_{\mu^\prime}-\frac{f^{\gamma}}{f^{\mu}}\mathbf{e}_{\mu}.
$$
The sum of all such terms over $e\in \supp(\gamma)$ collapses as a telescoping series since $\gamma$ is closed.
\end{proof}

For $1\leq k\leq n$, let $\Gamma_k$ denote the spanning subgraph of $\Gamma$ that has vertex set $\{1,\dots, m\}$, and which includes only the first $k$ edges of $\Gamma$ (see Figure~\ref{fig:mainproof1}(a) below for the case $k=m+3$). Clearly $\Gamma=\Gamma_n$. Let $\gamma=(\mu_1,\dots,\mu_\ell,\mu_1)$ be a circuit in $\Gamma_k$ for some $k$. A \emph{chord} of $\gamma$ in $\Gamma_k$ is any edge of the form $c=(\mu_r,\mu_s)$ for some $1 \leq r < s \leq \ell$ that does not lie in $\supp(\gamma)$. Every such chord $c$ splits $\gamma$ into two circuits: 
\begin{equation}
\label{eqn:splitting}
 \gamma_1=(\mu_{r},\dots,\mu_{s},\mu_{r}) \quad \text{and} \quad\gamma_2=(\mu_{1},\dots,\mu_{r},\mu_{s},\dots,\mu_{\ell},\mu_{1}).
\end{equation}
A circuit must have length at least four if it is to admit a chord. We define a {\it minimal circuit} of $\Gamma_k$ to be a circuit of length at least three that has no chords. 

\begin{lemma}
\label{lem:chord}
Let $\gamma$ be a circuit in $\Gamma_k$ admitting a chord in $\Gamma_k$ that splits $\gamma$ into circuits $\gamma_1$ and $\gamma_2$ as in \eqref{eqn:splitting}. Then the syzygy $\gbf{\sigma}_{\gamma}$ is contained in the module generated by $\gbf{\sigma}_{\gamma_1}$ and $\gbf{\sigma}_{\gamma_2}$.
\end{lemma}

\begin{proof}
Let $c$ be the chord. For $i=1,2$, let $\gamma_i\setminus c$ denote the walk obtained from $\gamma_i$ by removing the edge $c$. Since $\sign_{\gamma_1}(c)=-\sign_{\gamma_2}(c)$ we may rewrite
\begin{eqnarray*}
\gbf{\sigma}_{\gamma} & = & \sign_{\gamma_1}(c)\frac{f^\gamma}{f^c}\gbf{\varepsilon}_c + \sign_{\gamma_2}(c) \frac{f^\gamma}{f^c}\gbf{\varepsilon}_c + \sum_{e\in \supp(\gamma)} \sign_\gamma(e) \frac{f^{\gamma}}{f^{e}} \gbf{\varepsilon}_{e} \\
 & = & \sign_{\gamma_1}(c) \frac{f^{\gamma}}{f^{c}} \gbf{\varepsilon}_{c}  + \!\!\sum_{e\in \supp(\gamma_1\setminus c)} \sign_{\gamma_1}(e) \frac{f^{\gamma}}{f^{e}} \gbf{\varepsilon}_{e} + \sign_{\gamma_2}(c) \frac{f^{\gamma}}{f^{c}} \gbf{\varepsilon}_{c}  + \!\!\sum_{e\in \supp(\gamma_2\setminus c)} \sign_{\gamma_2}(e) \frac{f^{\gamma}}{f^{e}} \gbf{\varepsilon}_{e} \\
 & = & \frac{f^\gamma}{f^{\gamma_1}} \gbf{\sigma}_{\gamma_1} +  \frac{f^\gamma}{f^{\gamma_2}} \gbf{\sigma}_{\gamma_2}.
\end{eqnarray*}
It remains to note that $f^{\gamma_1}= f^{\mu_{r},\dots,\mu_{s}}$ divides $f^\gamma=f^{\mu_1,\dots,\mu_\ell}$, and similarly, $f^{\gamma_2}$ divides $f^\gamma$.
\end{proof}

We are now in a position to establish the main result of this section.

\begin{theorem}
\label{thm:main1}
For $1\leq k\leq n$, the $S$-module $\syz(F^k)$ is generated by the syzygies $\gbf{\sigma}_{\gamma}$, where $\gamma$ is a minimal circuit of $\Gamma_k$.
\end{theorem}
\begin{proof}
 We distinguish three cases. The first case, in which $1\leq k \leq m-1$,  is straightforward:  the graph $\Gamma_k$ admits no circuits and $\syz(F^k)=0$ by Lemma~\ref{lem:cyclic}, so the result holds.  
 
We prove the second case, in which $m\leq k\leq 2m-3$, by induction. For $k=m$, Lemma~\ref{lem:cyclic} shows that the $S$-module $\syz(F^m)$ is free with basis $\gbf{\sigma}_0$ from \eqref{eqn:sigma0}. The syzygy $\gbf{\sigma}_{\gamma_0}$ associated to the unique minimal circuit $\gamma_0 = (1,2,\dots,m,1)$ in $\Gamma_m$ coincides with $\gbf{\sigma}_0$, so the statement holds for $k=m$. Assume the statement for $\Gamma_{k-1}$ for any $m+1\leq k \leq 2m-3$, and let 
 \[
 \gbf{\sigma} = \sum_{j=1}^k s_j \gbf{\varepsilon}_{j}
 \]
 be a syzygy on $\gbf{\beta}_1,\dots, \gbf{\beta}_{k}$ where $s_1,\dots,s_k\in S$. 
 
 As a first step we reduce to the case in which the coefficients satisfy $s_j=0$ for $k-m+2\leq j\leq m$ (these indices determine the edges to the left of $\gbf{\beta}_k$ in Figure~\ref{fig:mainproof1}(a)).
  \begin{figure}[!ht]
    \centering
      \subfigure[]{
     \centering
 \psset{unit=1.2cm}
 \begin{pspicture}(0,-0.2)(3.4,3.7)
 \psset{linecolor=black}
 \cnodeput(1,0){A}{\tiny{$1$}}
 \cnodeput(2.4142,0){B}{\tiny{$2$}}
 \cnodeput(3.4142,1){C}{\tiny{$3$}} 
 \cnodeput(3.4142,2.4142){D}{\tiny{$4$}}
 \cnodeput(2.4142,3.4142){E}{\tiny{$5$}}
 \cnodeput(1,3.4142){F}{\tiny{$6$}}
  \cnodeput(0,2.4142){G}{\tiny{$7$}}
   \cnodeput(0,1){H}{\tiny{$m$}}  
  \ncline{->}{A}{B}
 \ncline{->}{B}{C}
 \ncline{->}{C}{D}
 \ncline{->}{D}{E}
  \ncline{->}{A}{C}
 \ncline{->}{A}{D}
 \ncline{->}{E}{F}
 \ncline{->}{F}{G}
 \ncline[linestyle=dashed]{->}{G}{H}
 \ncline{->}{A}{H}
 \ncline{->}{A}{E}\lput*{:290}{\small{$\gbf{\beta}_{k}$}}
  \end{pspicture}       }
      \qquad  \qquad
      \subfigure[]{
              \centering
 \psset{unit=1.2cm}
 \begin{pspicture}(0,-0.2)(3.4,3.7)
 \psset{linecolor=black}
 \cnodeput(1,0){A}{\tiny{$1$}}
 \cnodeput(2.4142,0){B}{\tiny{$2$}}
 \cnodeput(3.4142,1){C}{\tiny{$3$}} 
 \cnodeput(3.4142,2.4142){D}{\tiny{$4$}}
 \cnodeput(2.4142,3.4142){E}{\tiny{$5$}}
 \cnodeput(1,3.4142){F}{\tiny{$6$}}
  \cnodeput(0,2.4142){G}{\tiny{$7$}}
   \cnodeput(0,1){H}{\tiny{$m$}}  
   \put(0.9,1.9){$\gamma_1$}
       \put(2.2,1.9){$\gamma_2$}
 \ncline{->}{A}{B}
 \ncline{->}{B}{C}
 \ncline{->}{C}{D}
  \ncline{->}{A}{C}
 \psset{linecolor=black,linewidth=0.04}
 \ncline{->}{E}{F}
 \ncline{->}{F}{G}
 \ncline[linestyle=dashed]{->}{G}{H}
 \ncline{->}{A}{H}
 \ncline{->}{A}{E}
 \ncline{->}{D}{E}
 \ncline{->}{A}{D}
  \end{pspicture}
        }
    \caption{The graph $\Gamma_k$ for $m\leq k\leq 2m-3$ illustrated for $k=m+3$}
  \label{fig:mainproof1}
  \end{figure}
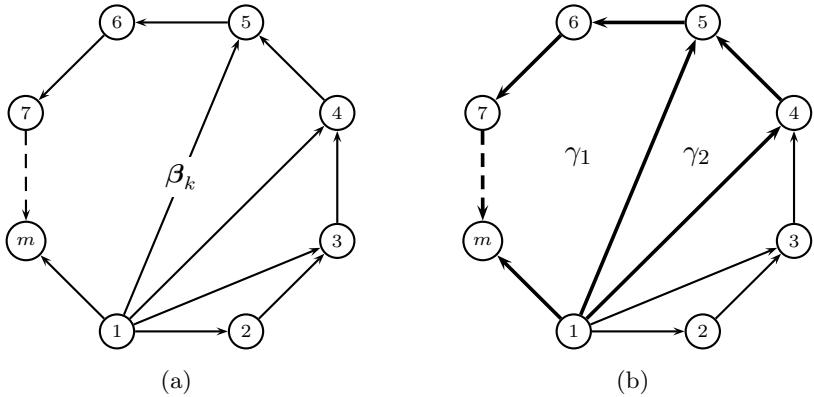
 Indeed, suppose otherwise, so $s_i\neq 0$ for some $k-m+2\leq i\leq m$.  By comparing the coefficient of $\mathbf{e}_\mu$ for each index $k-m+3\leq \mu \leq m$ in the equation
\[
0 = \psi(\gbf{\sigma})=  \sum_{j=1}^{k} s_j \bigg(\frac{f^{\mu_j,\nu_j}}{f^{\nu_j}}\mathbf{e}_{\nu_j}-\frac{f^{\mu_j,\nu_j}}{f^{\mu_j}}\mathbf{e}_{\mu_j}\bigg),
\]
we obtain a collection of equations
\begin{equation}
\label{eqn:comparecoeffstwo}
s_{k-m+2}f^{k-m+2,k-m+3} = s_{k-m+3}f^{k-m+3,k-m+4} = \cdots = s_{m-1}f^{m-1,m} = - s_m f^{1,m}
\end{equation}
which imply that $s_j\neq 0$ for all $k-m+2\leq j\leq m$. As illustrated in Figure~\ref{fig:mainproof1}(b) for $k=m+3$, the circuit $\gamma_1 := (1,k-m+2,k-m+3, \dots, m, 1)$ is minimal in $\Gamma_k$, and it determines both the monomial $f^{\gamma_1}= f^{1, k-m+2,k-m+3,\dots,m}$ and the syzygy
\begin{equation}
\label{eqn:sigma1}
\gbf{\sigma}_{\gamma_1} = -\frac{f^{\gamma_1}}{f^{1,m}}\gbf{\varepsilon}_m + \frac{f^{\gamma_1}}{f^{1,k-m+2}}\gbf{\varepsilon}_k + \sum_{j=k-m+2}^{m-1} \frac{f^{\gamma_1}}{f^{j,j+1}}\gbf{\varepsilon}_j.
\end{equation}
Equations \eqref{eqn:comparecoeffstwo} and the fact that $s_m\neq 0$ imply that $f^{\gamma_1}$ divides $s_m f^{1,m}$, and a straightforward computation shows that 
\[
\gbf{\sigma}_1:= \gbf{\sigma} + \frac{s_mf^{1,m}}{f^{\gamma_1}} \gbf{\sigma}_{\gamma_1} =  \bigg(s_k+\frac{s_mf^{1,m}}{f^{1,k-m+2}}\bigg)\gbf{\varepsilon}_k + \sum_{j=1}^{k-m+1} s_j\gbf{\varepsilon}_j + \sum_{j=m+1}^{k-1} s_j\gbf{\varepsilon}_j.
\]
In particular, if we expand $\gbf{\sigma}_1 =  \sum_{j=1}^k t_j \gbf{\varepsilon}_{j}$ for 
$t_1,\dots,t_k\in S$, then $t_j=0$ for $k-m+2\leq j\leq m$, and it suffices to prove the result for $\gbf{\sigma}_1$ as claimed. 

The second step is to repeat the above, comparing the coefficient of $\mathbf{e}_{k-m+2}$ in the equation $\psi(\gbf{\sigma}_1)=0$, and since $t_{k-m+2}=0$ we obtain 
\begin{equation}
\label{eqn:comparecoeffsthree}
t_{k-m+1}f^{k-m+1,k-m+2} + t_{k}f^{1,k-m+2}  = 0.
\end{equation}
If $t_k\neq 0$ then the minimal circuit $\gamma_2 := (1,k-m+2, k-m+1, 1)$ in $\Gamma_k$ from Figure~\ref{fig:mainproof1}(b) determines both the monomial $f^{\gamma_2}= f^{1, k-m+1,k-m+2}$ and the syzygy
\begin{equation}
\label{eqn:sigma2}
\gbf{\sigma}_{\gamma_2} = \frac{f^{\gamma_2}}{f^{1,k-m+2}}\gbf{\varepsilon}_k - \frac{f^{\gamma_2}}{f^{k-m+1,k-m+2}}\gbf{\varepsilon}_{k-m+1} -  \frac{f^{\gamma_2}}{f^{1,k-m+1}}\gbf{\varepsilon}_{k-1}.
\end{equation}
Equation \eqref{eqn:comparecoeffsthree} implies that $f^{\gamma_2}$ divides $t_{k}f^{1,k-m+2}$ and again, a straightforward computation, this time using equation \eqref{eqn:comparecoeffsthree}, shows that the coefficients of both $\gbf{\varepsilon}_k$ and $\gbf{\varepsilon}_{k-m+1}$ in the syzygy
\[
\gbf{\sigma}_2:= \gbf{\sigma}_1 - \frac{t_k f^{1,k-m+2}}{f^{\gamma_2}} \gbf{\sigma}_{\gamma_2} 
\]
are zero. This means that $\gbf{\sigma}_2\in \syz(F^{k-1})$, and we deduce from the inductive hypothesis that $\gbf{\sigma}_2$ is generated by the elements $\gbf{\sigma}_{\gamma}$ associated to minimal circuits $\gamma$ in $\Gamma_{k-1}$. 
Among all minimal circuits in $\Gamma_{k-1}$, only $\gamma = (1,k-m+1, k-m+2,\dots, m, 1)$ is not minimal in $\Gamma_{k}$; indeed, the edge labelled $\gbf{\beta}_k$ is a chord. However, this edge splits $\gamma$ into the circuits $\gamma_1, \gamma_2$ defined earlier in the current proof that \emph{are} minimal in $\Gamma_k$, and Lemma~\ref{lem:chord} writes $\gbf{\sigma}_\gamma$ as an $S$-linear combination of $\gbf{\sigma}_{\gamma_1}$ and $\gbf{\sigma}_{\gamma_2}$. Thus, the syzygy $\gbf{\sigma}_2$, and hence both $\gbf{\sigma}_1$ and $\gbf{\sigma}$, are generated by the elements $\gbf{\sigma}_{\gamma}$ associated to minimal circuits $\gamma$ in $\Gamma_{k}$. This completes the proof for $m\leq k\leq 2m-3$.

Finally, consider $2m-2\leq k\leq n$. Given any monomial order on $S$, let $>$ denote the term over position order on the free $S$-module $\bigoplus_{\mu=1}^m S\textbf{e}_\mu$, that is, $>$ is the monomial order defined for $g, g^\prime\in S$ and $1\leq \mu, \nu\leq m$ by taking $g^\prime \mathbf{e}_{\nu}>g \mathbf{e}_{\mu}$ if and only if $g^\prime f^{\nu} >g f^{\mu} $ with respect to the monomial order on $S$, or $g^\prime f^{\nu} =g f^{\mu}$ and  $\nu > \mu$. It follows that for $1 \leq j \leq k$, the leading term of $\gbf{\beta}_j$ with respect to this order is $f^{\mu_j,\nu_j}/f^{\nu_j} \mathbf{e}_{\nu_j}$. This implies that the S\emph{-vectors of critical pairs} are the elements
$$
\mathrm{S}(\gbf{\beta}_{i},\gbf{\beta}_{j})=\frac{f^{\mu_i,\mu_j,\nu_j}}{f^{\mu_j,\nu_j}}\gbf{\beta}_{j}-\frac{f^{\mu_i,\mu_j,\nu_j}}{f^{\mu_i,\nu_j}}\gbf{\beta}_{i}
$$
arising from all elements in $\mathbb{B}_k:=\{ (i,j) \mid 1 \leq i < j \leq k, \nu_{i}=\nu_{j}\}$  (see Kreuzer--Robbiano~\cite[Definition~2.5.1]{KR00}). Substituting \eqref{eqn:Beta} into every S-vector ensures that the leading terms cancel by definition. Since any critical pair $(i,j)$ corresponds to a pair of directed edges $(\mu_i,\nu_j)$ and $(\mu_j,\nu_j)$ in $\Gamma_k$, 
 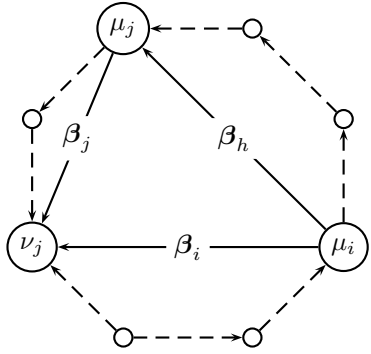
\begin{figure}[!ht]
 \centering
 \psset{unit=1.2cm}
 \begin{pspicture}(0,0)(3.4,3.7)
 \psset{linecolor=black}
 \cnodeput(1,0){A}{}
 \cnodeput(2.4142,0){B}{}
 \cnodeput(3.4142,1){C}{\small{$\mu_i$}} 
 \cnodeput(3.4142,2.4142){D}{}
 \cnodeput(2.4142,3.4142){E}{}
 \cnodeput(1,3.4142){F}{\small{$\mu_j$}}
  \cnodeput(0,2.4142){G}{}
   \cnodeput(0,1){H}{\small{$\nu_j$}}  
  \ncline[linestyle=dashed]{->}{A}{B} 
   \ncline[linestyle=dashed]{->}{B}{C}
    \ncline[linestyle=dashed]{->}{C}{D}
     \ncline[linestyle=dashed]{->}{D}{E}
      \ncline[linestyle=dashed]{->}{E}{F}
       \ncline[linestyle=dashed]{->}{F}{G} 
       \ncline[linestyle=dashed]{->}{G}{H}
        \ncline[linestyle=dashed]{->}{A}{H}      
  \ncline{->}{C}{H}\lput*{:180}{\small{$\gbf{\beta}_i$}}
   \ncline{->}{C}{F}\lput*{:-135}{\small{$\gbf{\beta}_h$}}
    \ncline{->}{F}{H}\lput*{:112.5}{\small{$\gbf{\beta}_j$}}
  \end{pspicture}
 \caption{Minimal circuit in $\Gamma_k$ for $2m-2\leq k\leq n$ where $i<j$.}
  \label{fig:triangle}
 \end{figure}
the S-vector can then be written as a multiple of the generator $\gbf{\beta}_{(\mu_i,\mu_j)}$ corresponding to the third directed edge from Figure~\ref{fig:triangle}. Indeed, if we choose the index $1\leq h\leq k$ so that $\gbf{\beta}_h =  \gbf{\beta}_{(\mu_i,\mu_j)}$, then we compute explicitly that the `standard expressions' are
\[
\mathrm{S}(\gbf{\beta}_{i},\gbf{\beta}_{j})=-\frac{f^{\mu_i, \mu_j, \nu_i}}{f^{\mu_i,\mu_j}}\gbf{\beta}_{h}. 
\]
Moreover, we deduce from Buchberger's Criterion \cite[Theorem 15.8]{Eisenbud95} that $\gbf{\beta}_1,\dots, \gbf{\beta}_k$ are a Gr\"{o}bner basis of $F^k$. Every standard expression determines a syzygy, namely
\begin{equation}
\label{eqn:sigmaij}
\gbf{\sigma}_{(i,j)}=\frac{f^{\mu_i,\mu_j,\nu_j}}{f^{\mu_j,\nu_j}}\gbf{\varepsilon}_{j}-\frac{f^{\mu_i,\mu_j,\nu_j}}{f^{\mu_i,\nu_j}}\gbf{\varepsilon}_{i} + \frac{f^{\mu_i, \mu_j, \nu_i}}{f^{\mu_i,\mu_j}}\gbf{\varepsilon}_{h}. 
\end{equation}
Schreyer's theorem~\cite[Theorem 15.10]{Eisenbud95} implies that the set of syzygies $\{\gbf{\sigma}_{(i,j)} \mid (i,j)\in \mathbb{B}_k\}$ is a system of generators for $\syz(F^k)$. Let $\gamma(i,j):=(\mu_i, \mu_j, \nu_j,\mu_i)$ denote circuit in $\Gamma_k$ obtained by traversing the edges labelled $\gbf{\beta}_h$, $\gbf{\beta}_j$ according to orientation followed by the edge labelled $\gbf{\beta}_i$ against orientation (see Figure~\ref{fig:triangle}). Then $\gbf{\sigma}_{(i,j)}$ coincides with the syzygy $\gbf{\sigma}_{\gamma(i,j)}$ from Lemma~\ref{lem:syzygy}, and the result is a consequence of the following Lemma.
\end{proof}

\begin{lemma}
\label{lem:mincircuits}
For $2m-3\leq k\leq n$, the minimal circuits in the graph $\Gamma_k$ are precisely those of the form $\gamma(i,j)= (\mu_i, \mu_j, \nu_j,\mu_i)$ arising from pairs $(i,j)$ in $\mathbb{B}_k=\{ (i,j) \mid 1 \leq i < j \leq k, \nu_{i}=\nu_{j}\}$.
\end{lemma}
\begin{proof}
We proceed by induction. Let $\gamma$ be a minimal circuit in $\Gamma_{2m-3}$ that is not of the form $\gamma(i,j)$ for any $(i,j)\in\mathbb{B}_{2m-3}$. Since $\gamma$ is a circuit, it must traverse an edge $e$ of the subgraph $\Gamma_m$, and since $\gamma\neq \gamma(i,j)$, then either the edge that follows $e$ in $\gamma$, or that preceding $e$ in $\gamma$, must lie in $\Gamma_m$. In either case, $\gamma$ traverses two edges from $\Gamma_m$ that share a common vertex $\mu$. The special nature of $\Gamma_{2m-3}$ then forces the edge $(1,\mu)$ to be a chord of $\gamma$, a contradiction. Assume now that the result holds for $\Gamma_{k-1}$ and let $\gamma$ be a minimal circuit in $\Gamma_{k}$ that is not of the form $\gamma(i,j)$ for any $(i,j)\in\mathbb{B}_k$. If the edge $e_k=(\mu_k,\nu_k)$ does not lie in $\supp(\gamma)$ then the result holds by induction, so we suppose otherwise. Let $e$ be the unique edge in $\supp(\gamma)\setminus \{e_k\}$ that has $\nu_k$ as a vertex. There are three cases: 
\begin{enumerate}
\item[\one] $e=(\nu_k-1,\nu_k)$, in which case $(\mu_k,\nu_{k}-1)$ is a chord because $\gamma\neq \gamma(\nu_{k}-1,k)$;  
\item[\two] $e=(\nu_k,\nu_k+1)$, in which case $\gamma$ must pass through a vertex of the form $1\leq \mu \leq \mu_k$ since it is a circuit, but then $(\mu,\nu_k)$ is a chord;
\item[\three] $e=(\mu,\nu_k)$ for some $1\leq \mu < \mu_k$. Since $\gamma\neq \gamma(j,k)$ for any $j< k$, the circuit $\gamma$ must pass through another vertex of the form $1\leq \mu^\prime < \mu_k$, but then $(\mu^\prime,\nu_k)$ is a chord.
\end{enumerate}
Thus, the minimal circuit $\gamma$ cannot exist. 
\end{proof}

\begin{remark} 
\begin{enumerate}
\item If for $2m-2\leq k\leq n$ we draw the vertices of $\Gamma_k$ spaced evenly around a circle centred at the origin in $\RR^2$, then each minimal circuit $\gamma$ has length three and hence determines a triangle as in Figure~\ref{fig:triangle}. In the spirit of the Taylor resolution of a monomial ideal (see, for example, Bayer--Peeva--Sturmfels \cite{BPS98}), the triangle can be viewed as a 2-cell that defines $f^{\mu_i,\mu_j,\nu_j}$, and the edges are 1-cells defining $f^{\mu_i,\mu_j}, f^{\mu_i,\nu_j}$ and $f^{\mu_j,\nu_j}$. The coefficients of the syzygy $\gbf{\sigma}_{(i,j)}$ are then simply the quotients of the monomial for the 2-cell divided by the monomial for the corresponding 1-cell. An analogous statement holds for $m\leq k\leq 2m-3$, where the syzygies $\gbf{\sigma}_0$ and $\gbf{\sigma}_1$ from the proofs of Lemma~\ref{lem:cyclic} and Theorem~\ref{thm:main1} respectively define polygons with more than three sides. 
\item We emphasise that our choice of order on  the set of transpositions of $m$ letters is imposed on us by the geometry: the filtration in Proposition~\ref{prop:filtration2} below requires that  the $S$-module $F^k$ contains $F^0$ for $1\leq k\leq n$. Without this constraint one could choose an alternative order in which each minimal circuit of $\Gamma_k$ for $m \leq k \leq n$ determines a triangle, leading to a more unified proof of Theorem~\ref{thm:main1}. Indeed, since $f^1,\dots, f^m$ are monomials, the modules $\syz(F^k)$ can be read off directly from the Taylor resolution for $1\leq k\leq m$.
\end{enumerate}
\end{remark}

As an application of Theorem~\ref{thm:main1}, we introduce a filtration of the module $S$-module $\ker(\varphi)=\syz(M)$ that feeds into the proof of our main result. For $1\leq k\leq n$, the $S$-modules $F^k$ define a filtration 
\[
0\subseteq F^1\subseteq F^2\subseteq \cdots \subseteq F^{n-1}\subseteq F^n=\syz(M)
\]
in which the successive quotients are cyclic $S$-modules
\begin{equation}
\label{eqn:betakquotient}
\frac{F^k}{F^{k-1}} \cong  \frac{\langle \gbf{\beta}_k \rangle}{\langle\gbf{\beta}_1,\dots, \gbf{\beta}_{k-1} \rangle \cap \langle \gbf{\beta}_k \rangle}.
\end{equation}
The next result gives an explicit description of these quotient modules.

\begin{proposition}
\label{prop:filtration1}
For each $1\leq k\leq n$, the quotient $F^k/F^{k-1}$ is isomorphic to the cyclic $S$-module $S/I_k$, where the monomial ideal $I_k$ depends on $k$ as follows:
  \begin{enumerate}
\item[\one] for $1\leq k\leq m-1$, the ideal $I_k$ is the zero ideal;
\item[\two] for $k=m$, the ideal $I_k$ is principal with generator $f^{1,\dots,m}/f^{1,m}$;
\item[\three] for $m+1\leq k\leq 2m-3$, the ideal is
\[
I_k=\bigg\langle \frac{f^{1,k-m+2,k-m+3,\dots, m}}{f^{1,k-m+2}}, \frac{f^{1,k-m+1,k-m+2}}{f^{1,k-m+2}}\bigg\rangle;
\]
\item[\four] for $2m-2 \leq k\leq n$, the corresponding transposition is $\tau_k=(\mu_k,\nu_k)$, and the ideal is
\[
I_k = \bigg\langle \frac{f^{\mu,\mu_k,\nu_k}}{f^{\mu_k,\nu_k}} \: \bigg\vert \: \mu\in \{1, \dots, \mu_k-1\}\cup\{\nu_k-1\}\bigg\rangle.
\]
  \end{enumerate}
\end{proposition}
\begin{proof}
For $1 \leq k \leq n$, let $\{\gbf{\sigma}_1,\dots, \gbf{\sigma}_r\}$ be a set of generators for the $S$-module $\syz(F^k)$. If we write $\gbf{\sigma}_{\nu}=\sum_{j=1}^k s_{\nu j}\gbf{\varepsilon}_j$ with $s_{\nu 1},\dots,s_{\nu k} \in S$ for $1 \leq \nu \leq r$, then \cite[Proposition~3.2.3]{KR00} implies that the coefficients $s_{1k},\dots, s_{rk}$ of $\gbf{\varepsilon}_k$ give the generators $s_{1k}\gbf{\beta}_k,\dots, s_{rk}\gbf{\beta}_k$ of the $S$-module $\langle\gbf{\beta}_1,\dots, \gbf{\beta}_{k-1} \rangle \cap \langle\gbf{\beta}_k \rangle$, so we obtain   
\[
\frac{F^k}{F^{k-1}} \cong \frac{S}{\langle s_{1k},\dots, s_{rk} \rangle}. 
\]
It remains to compute $I_k:=\langle s_{1k},\dots, s_{rk} \rangle$. Parts \one\ and \two\ now follow from Lemma~\ref{lem:cyclic} and equation \eqref{eqn:sigma0}. For part \three, the proof of Theorem~\ref{thm:main1} shows that the only minimal circuits $\gamma$ in $\Gamma_k$ with $m+1\leq k\leq 2m-3$ for which the associated syzygy $\gbf{\sigma}_\gamma$ has a nonzero coefficient for $\gbf{\varepsilon}_k$ are $\gamma_1 := (1,k-m+2,k-m+3, \dots, m, 1)$ and $\gamma_2 := (1,k-m+2, k-m+1, 1)$. These nonzero coefficients are presented in equations \eqref{eqn:sigma1} and \eqref{eqn:sigma2}, namely
\[
\frac{f^{\gamma_1}}{f^{1,k-m+2}} = \frac{f^{1,k-m+2,k-m+3,\dots, m}}{f^{1,k-m+2}}\quad\text{and}\quad \frac{f^{\gamma_2}}{f^{1,k-m+2}}= \frac{f^{1,k-m+1,k-m+2}}{f^{1,k-m+2}}.
\]
For part \four, we deduce from Theorem~\ref{thm:main1} and Lemma~\ref{lem:mincircuits} that $\syz(F^k)$ is generated by the syzygies $\gbf{\sigma}_{(i,j)} = \gbf{\sigma}_{\gamma(i,j)}$ associated to pairs $(i,j)\in \mathbb{B}_k$. By equation \eqref{eqn:sigmaij}, such syzygies have a nonzero coefficient of $\gbf{\varepsilon}_k$ if and only if $(i,j)=(i,k)$ for those $1\leq i < k$ satisfying $\nu_i=\nu_k$. The $i$th edge $(\mu_i,\nu_i)$ in $\Gamma_k$ has $\nu_i=\nu_k$ if and only if $\mu_i \in \{1, \dots, \mu_k-1\}\cup\{\nu_k-1\}$, that is, we must consider all pairs of the form $(\mu,\nu_k)$ for $\mu\in \{1, \dots, \mu_k-1\}\cup\{\nu_k-1\}$. Equation \eqref{eqn:sigmaij} shows that the coefficient of $\gbf{\varepsilon}_k$ in this case is $f^{\mu,\mu_k,\nu_k}/f^{\mu_k,\nu_k}$ as required.
\end{proof}

\begin{remark}
The generators of $I_k$ listed in Proposition~\ref{prop:filtration1} need not be minimal for $m+1\leq k\leq n$. For example (though not the simplest), a straightforward calculation for the module $M$ over $S=\kk[x_1,\dots, x_7]$ with generators 
\[
f^1=x_1x_6,\; f^2= x_1x_2x_7,\; f^3=x_2x_3, \; f^4=x_3x_4, \; f^5=x_4x_5x_7, \; f^6= x_5x_6
\]
gives $I_{k}=S$ for $k=9,10,12,13$. Thus, $I_k$ is principal even though this ideal is listed as having more than one generator in Proposition~\ref{prop:filtration1}.
\end{remark}

\section{Cohomology of wheels on toric varieties}
\label{sec:cohomologyWheels}
Let $X$ be a normal variety over $\kk$. The divisor class group $\Cl(X)$ is defined to be the group of linear equivalence classes of Weil divisors on $X$. Since $X$ is normal, two divisors $D$ and $D^\prime$ are linearly equivalent if and only if the associated rank-one reflexive sheaves $\mathscr{O}_X(D)$ and $\mathscr{O}_X(D^\prime)$ are isomorphic. We may therefore identify elements of the class group of $X$ with (isomorphism classes of) sheaves of the form $\mathscr{O}_X(D)$. In particular, for a Cartier divisor $D$ on $X$ defining an invertible sheaf $L:=\mathscr{O}_X(D)$, we sometimes write $L\in \Cl(X)$.

Let $X$ be a normal toric variety over $\kk$ defined by a fan $\Sigma$ in the real vector space $N\otimes_\ZZ \RR$ with underlying lattice $N$ of rank $n$. Write $\Sigma(1)$ for the set of one-dimensional cones in $\Sigma$, set $d:=\vert \Sigma(1)\vert$, and let $v_\rho\in N$ denote the primitive lattice point on the cone $\rho$. Each $\rho\in \Sigma(1)$ determines a torus-invariant Weil divisor $D_\rho$ in $X$, and we let $\ZZ^d$ denote the free abelian group of torus-invariant Weil divisors. Assume that $X$ has no torus factors. The map $\deg\colon \ZZ^d\to \Cl(X)$ sending $D$ to the sheaf $\mathscr{O}_X(D)$ fits into a short exact sequence of abelian groups
 \[
 0 \xlongrightarrow{} M \xlongrightarrow{\div}  \ZZ^d \xlongrightarrow{\deg} \Cl(X)\xlongrightarrow{} 0,
 \]
 where $M$ is the lattice dual to $N$ and where $m\in M$ maps to $\div(m)=\sum_{\rho\in \Sigma(1)} \langle m,v_\rho\rangle D_\rho$. The restriction of the map $\deg\colon \ZZ^d\to \Cl(X)$ to the subsemigroup $\NN^d$ defines a $\Cl(X)$-grading of the \emph{Cox ring} of $X$ which is the semigroup ring $S:= \kk[x_1,\dots,x_d]$ of $\NN^d$. Explicitly, the degree of a monomial $\prod_{\rho\in \Sigma(1)} x_\rho^{a_\rho}\in S$ is $\mathscr{O}_X(\sum_{\rho\in \Sigma(1)} a_\rho D_\rho)\in \Cl(X)$. Armed with this $\Cl(X)$-grading of the ring $S$, Cox~\cite[Proposition~3.1]{Cox95} introduced an exact covariant functor 
\begin{equation}
\label{eq:Cox}
\{\Cl(X)\text{-graded }S\text{-modules}\} \longrightarrow \{\text{quasicoherent }\mathscr{O}_X\text{-modules}\} \; :\; F\longmapsto \widetilde{F}
\end{equation}
from the category of $\Cl(X)$-graded $S$-modules to the category of quasi-coherent sheaves on $X$, and Musta{\c{t}}{\u{a}}~\cite[Theorem~1.1]{Mustata02} subsequently showed that the functor is essentially surjective, i.e., that every quasi-coherent sheaf (up to isomorphism) on $X$ lies in the image of this functor. If $X$ is smooth, two such graded modules determine isomorphic sheaves if and only if they agree upto saturation by Cox's irrelevant ideal $B = (\prod_{\rho\not\subset \sigma} x_\rho \mid \sigma\in \Sigma)$, but we do not use this fact (until Remark~\ref{rem:hex}). The important point for us is that the functor enables us to lift a complex of quasi-coherent sheaves on $X$ to obtain a complex of $\Cl(X)$-graded $S$-modules which we can study, and then push down again to the original complex of sheaves. 

\medskip

As described in the introduction, our primary motivation is to study four-term complexes $T^\bullet$ on $X$ of the form  
\eqref{eqn:Tbullet} for some integer $m \geq 2$. In fact, we take as the primary object of study the corresponding diagram of torus-equivariant maps between invertible sheaves on $X$: 
\begin{equation}
\label{eqn:diagrammaintext}
\begin{split}
    \centering    
         \psset{unit=0.45cm}
     \begin{pspicture}(0,-1)(25,13.7)
\cnodeput*(0,6){A}{$L$} 
\cnodeput*(8,12){B}{$L_{1,2}$}
\cnodeput*(8,9){C}{$L_{2,3}$} 
\cnodeput*(8,6){D}{$L_{3,4}$}
\cnodeput*(8,3.2){S}{$\vdots$}
\cnodeput*(8,0){E}{$L_{m,1}$}
\cnodeput*(18,12){F}{$L_{1}$}
\cnodeput*(18,9){G}{$L_{2}$}
\cnodeput*(18,6){H}{$L_{3}$}
\cnodeput*(18,3.2){T}{$\vdots$}
\cnodeput*(18,0){I}{$L_{m}$}
\cnodeput*(26,6){J}{$L.$}
\psset{nodesep=1pt}
   \ncline{->}{A}{B}\lput*{:U}(0.6){$\scriptstyle{D_{1,2}}$}
   \ncline{->}{A}{C}\lput*{:U}(0.6){$\scriptstyle{D_{2,3}}$}
   \ncline{->}{A}{D}\lput*{:U}(0.6){$\scriptstyle{D_{3,4}}$}
   \ncline{->}{A}{E}\lput*{:U}(0.6){$\scriptstyle{D_{m,1}}$}
 \ncline{->}{B}{F}\lput*{:U}(0.4){$\scriptstyle{D^2_1}$}
  \ncline{->}{B}{G}\lput*{:U}(0.4){$\scriptstyle{D^1_2}$}
  \ncline{->}{C}{G}\lput*{:U}(0.4){$\scriptstyle{D^3_2}$}
  \ncline{->}{C}{H}\lput*{:U}(0.4){$\scriptstyle{D^2_3}$}
\ncline{->}{D}{H}\lput*{:U}(0.4){$\scriptstyle{D^4_3}$}
\ncline{->}{E}{I}\lput*{:U}(0.4){$\scriptstyle{D^1_m}$}
\nccurve[angleA=-40,angleB=140]{->}{E}{F}\lput*{:U}(0.4){$\scriptscriptstyle{D^m_1}$}
 \ncline{->}{F}{J}\lput*{:U}(0.4){$\scriptstyle{D^{1}}$}
   \ncline{->}{G}{J}\lput*{:U}(0.4){$\scriptstyle{D^{2}}$}
   \ncline{->}{H}{J}\lput*{:U}(0.4){$\scriptstyle{D^{3}}$}
   \ncline{->}{I}{J}\lput*{:U}(0.4){$\scriptstyle{D^{m}}$}
       \end{pspicture}
 \end{split}
 \end{equation}

\noindent Every torus-equivariant map is multiplication by a torus-invariant section of an invertible sheaf on $X$, and we illustrate on each arrow the torus-invariant Cartier divisor of zeros of the corresponding section. Thus, for example, the effective divisor $D^1_{2}\in H^0(L_2\otimes L_{1,2}^{-1})\cong \Hom(L_{1,2},L_2)$ denotes the Cartier divisor of zeros of the section that defines the map from $L_{1,2}$ to $L_2$. One can think of any such diagram as a representation of a quiver (arising as the skeleton of a three-dimensional rhombic polyhedron) in the category of invertible sheaves on $X$. 

Throughout, we impose relations on this quiver, whereby each of the two-dimensional rhombic faces of this quiver forms a commutative square, i.e.
\begin{align}
D^j_{j+1}+D^{j+1} & = D_j^{j+1}+D^j , \label{eqn:relations1}\\
D^{j-1}_j+D_{j-1,j} & = D_j^{j+1}+D_{j,j+1},\label{eqn:relations2}
 \end{align}
 for $1\leq j\leq m$ (working modulo $m$, with indices in the range $1,\dots, m$). We now describe how a diagram of the form  \eqref{eqn:diagrammaintext} gives rise to a complex of $\Cl(X)$-graded $S$-modules precisely when \eqref{eqn:relations1} and \eqref{eqn:relations2} hold. Indeed, let $S(L)$ denote the free $S$-module with generator $\textbf{e}_L$ in degree $L$, and for $1\leq j\leq m$ let $S(L_j)$ and $S(L_{j,j+1})$ denote the free $S$-modules with generators $\textbf{e}_j$ in degree $L_j$ and $\textbf{e}_{j,j+1}$ in degree $L_{j,j+1}$ respectively. In addition, let $f^j$, $f^j_{j+1}$, $f^{j+1}_j$, $f_{j,j+1}$ denote the monomials in the Cox ring $S$ whose divisors of zeroes are the torus-invariant Cartier divisors $D^j$, $D^j_{j+1}$, $D^{j+1}_{j}$, $D_{j,j+1}$ from \eqref{eqn:diagrammaintext}. Consider the sequence of $\Cl(X)$-graded $S$-modules
\begin{equation}
\label{eqn:complexSmods}
S(L) \xlongrightarrow{\varphi^3}  \bigoplus_{j=1}^m S(L_{j,{j+1}}) \xlongrightarrow{\varphi^2} \bigoplus_{j=1}^m S(L_j) \xlongrightarrow{\varphi^{1}} S(L),
\end{equation}
with maps
\[
\varphi^{3}(\textbf{e}_{L}) = \sum_{j=1}^m f_{j,j+1} \textbf{e}_{j,j+1}, \quad
\varphi^{2}(\textbf{e}_{j,j+1}) = f^j_{j+1}\textbf{e}_{j+1} - f^{j+1}_j \textbf{e}_j,\quad
\varphi^{1}(\textbf{e}_j) = f^j\textbf{e}_L.
\]
We claim that the sequence \eqref{eqn:complexSmods} is a complex if and only if the relations \eqref{eqn:relations1} and \eqref{eqn:relations2} hold. Indeed, \eqref{eqn:complexSmods} is a complex if and only if we have
\[
(\varphi^2 \circ \varphi^3)(\textbf{e}_{L})=0\quad \text{and}\quad (\varphi^1 \circ \varphi^2)(\textbf{e}_{j,j+1})=0 \text{ for }1\leq j\leq m,
\]
which is the case if and only if $f^j_{j+1}f^{j+1}-f^{j+1}_j f^j=0$ and $f^{j-1}_j f_{j-1,j}-f^{j+1}_j f_{j,j+1}=0$ for all $1\leq j\leq m$, and these equations hold if and only if \eqref{eqn:relations1} and \eqref{eqn:relations2} hold for all $1\leq j\leq m$.  In summary, the diagram \eqref{eqn:diagrammaintext} of invertible sheaves in which the relations \eqref{eqn:relations1} and \eqref{eqn:relations2} hold determines a complex of $\Cl(X)$-graded $S$-modules of the form \eqref{eqn:complexSmods}. Conversely, to any complex of the form \eqref{eqn:complexSmods}, one can reverse this procedure to obtain a diagram \eqref{eqn:diagrammaintext} of invertible sheaves on $X$ in which the relations \eqref{eqn:relations1} and \eqref{eqn:relations2} hold.

Applying the exact functor \eqref{eq:Cox} to the complex \eqref{eqn:complexSmods} of $\Cl(X)$-graded $S$-modules determines a complex $T^\bullet$ of locally free sheaves on $X$ of the form
$$
L \xlongrightarrow{d^3}  \bigoplus_{j=1}^m L_{j,{j+1}} \xlongrightarrow{d^2} \bigoplus_{j=1}^m L_j \xlongrightarrow{d^1} L,
$$
 where each differential is torus-equivariant, and where the right-hand copy of $L$ lies in degree zero. Moreover, for each $1\leq j\leq m$ the restriction of the differential $d^2$ to the summand $L_{j,j+1}$ has image in $L_j\oplus L_{j+1}$ (with indices modulo $m$). This is the \emph{total chain complex} $T^{\bullet}$ of the diagram \eqref{eqn:diagrammaintext}. The complexes studied by Cautis--Logvinenko~\cite{CL09}, Cautis--Craw--Logvinenko~\cite{CCL12} and Bocklandt--Craw--Quintero-V\'{e}lez~\cite{BCQ12} that motivated our main result all take this form. The invertible sheaves at the left and right of diagram \eqref{eqn:diagrammaintext} coincide, so the sheaves and the maps between them in diagram \eqref{eqn:diagrammaintext} can be represented equally well in a planar picture as in Figure~\ref{fig:wheel}; we call this the \emph{wheel} of invertible sheaves on $X$. 
\begin{figure}[!ht]
\centering    
         \psset{unit=0.45cm}
     \begin{pspicture}(-15,-8)(15,9.5)
\cnodeput*(0,0){A}{$L$} 
\cnodeput*(8,0){B}{$L_{1}$}
\cnodeput*(7,4){C}{$L_{1,2}$} 
\cnodeput*(4,7){D}{$L_{2}$}
\cnodeput*(0,8){E}{$L_{2,3}$}
\cnodeput*(-4,7){F}{$L_{3}$}
\cnodeput*(-7,4){G}{$L_{3,4}$}
\cnodeput*(-8,0){H}{$L_{4}$}
\cnodeput*(-7,-4){I}{\;}
\cnodeput*(-4,-7){J}{\;\;\;\;\;}
\cnodeput*(0,-8){K}{\;}
\cnodeput*(4,-7){L}{$L_{m}$}
\cnodeput*(7,-4){M}{$L_{m,1}$}
\psset{nodesep=1pt}
\ncline[linewidth=.6mm,linestyle=dotted,dotsep=15pt]{-}{I}{J}
\ncline[linewidth=.6mm,linestyle=dotted,dotsep=15pt]{-}{K}{J}  
   \ncline{<-}{A}{B}\lput*{:U}(0.6){$\scriptstyle{D^1}$}
   \ncline{<-}{A}{D}\lput*{:-60}(0.6){$\scriptstyle{D^2}$}
     \ncline{<-}{A}{F}\lput*{:-120}(0.6){$\scriptstyle{D^3}$}
       \ncline{<-}{A}{H}\lput*{:-180}(0.6){$\scriptstyle{D^4}$}
         \ncline{<-}{A}{L}\lput*{:60}(0.6){$\scriptstyle{D^m}$}
    \ncline{->}{A}{C}\lput*{:-30}(0.6){$\scriptstyle{D_{1,2}}$}
      \ncline{->}{A}{E}\lput*{:-90}(0.6){$\scriptstyle{D_{2,3}}$}
        \ncline{->}{A}{G}\lput*{:-150}(0.6){$\scriptstyle{D_{3,4}}$}
            \ncline{->}{A}{M}\lput*{:30}(0.6){$\scriptstyle{D_{m,1}}$}
    \ncline{<-}{B}{C}\bput*{:-105}{$\scriptstyle{D^2_1}$} 
     \ncline{->}{C}{D}\bput*{:-135}{$\scriptstyle{D^1_2}$}         
    \ncline{<-}{D}{E}\bput*{:-165}{$\scriptstyle{D^3_2}$}         
     \ncline{->}{E}{F}\bput*{:-195}{$\scriptstyle{D^2_3}$}           
 \ncline{<-}{F}{G}\bput*{:-225}{$\scriptstyle{D^4_3}$}  
     \ncline{->}{G}{H}\bput*{:-255}{$\scriptstyle{D^3_4}$}   
      \ncline{<-}{H}{I}  
 \ncline{->}{K}{L}    
     \ncline{<-}{L}{M}\bput*{:-45}{$\scriptstyle{D^1_m}$}   
      \ncline{->}{M}{B}  \bput*{:-75}{$\scriptstyle{D^m_1}$} 
        \end{pspicture}
    \caption{Wheel of invertible sheaves on $X$}
     \label{fig:wheel}
    \end{figure}
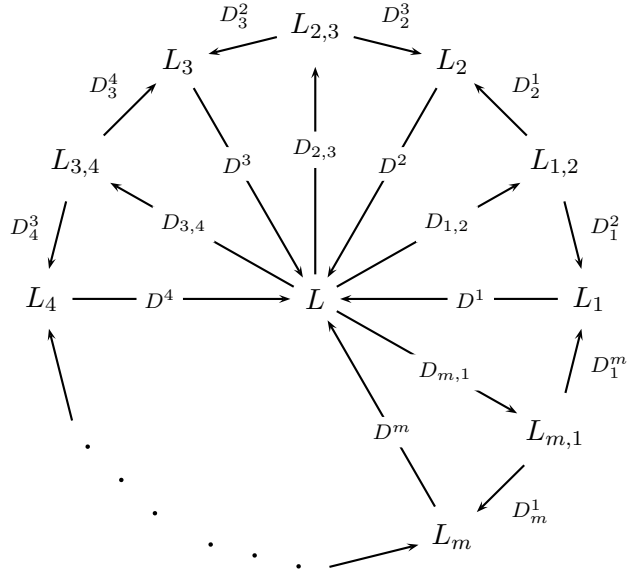

We now use the results of the previous section to compute the cohomology of the complex $T^{\bullet}$. For this purpose, we first note that the map $\varphi^1$ is of the form considered in the preceding section, so we may list the generators of its kernel in a sequence $\gbf{\beta}_1, \dots,\gbf{\beta}_n$ with $n=\binom{m}{2}$. We also list the generators of the image of $\varphi^2$ as 
\[
\gbf{\alpha}_j:=  f^j_{j+1}\textbf{e}_{j+1} - f^{j+1}_j \textbf{e}_j
\]  
for $1\leq j\leq m$. The next proposition is central to the main result of this paper.

\begin{proposition}
\label{prop:filtration2}
The $S$-modules
\[
F^k = \begin{cases} \langle \gbf{\beta}_1, \dots,\gbf{\beta}_k,\gbf{\alpha}_{k+1}, \dots,\gbf{\alpha}_m\rangle & \mbox{for } 1\leq k \leq m, \\  \langle \gbf{\beta}_1, \dots,\gbf{\beta}_m,\gbf{\beta}_{m+1}, \dots,\gbf{\beta}_j\rangle & \mbox{for } m+1\leq j\leq n, \end{cases}
\]
define a filtration 
\[
\im(\varphi^2)=F^0 \subseteq F^1\subseteq \cdots \subseteq F^{n-1}\subseteq F^n=\ker(\varphi^1).
\]
Moreover, for $1 \leq k \leq n$ and for the transposition is $\tau_k=(\mu_k,\nu_k)$, the quotient $F^k/F^{k-1}$ is isomorphic to the cyclic $\Pic(X)$-graded $S$-module $(S/I_k)(L_{\mu_k}\otimes L_{\nu_k}\otimes L^{-1}(\gcd(D^{\mu_k},D^{\nu_k})))$, where the monomial ideal $I_k$ depends on $k$ as follows:
\begin{enumerate}
\item[(1)]for $1 \leq k \leq m$, the ideal is
\[
I_k = \bigg\langle \gcd(f^k_{k+1},f^{k+1}_k), \frac{\lcm(f^{1,\dots,m},\gcd(f^{k+1}_{k+2},f^{k+2}_{k+1}),\dots,\gcd(f^{m}_{1},f^{1}_{m}) )}{f^{k,k+1}} \bigg\rangle;
\]
\item[(2)]for $m+1 \leq k \leq 2m-3$, the ideal is
\[
I_k=\bigg\langle \frac{f^{1,k-m+2,k-m+3,\dots, m}}{f^{1,k-m+2}}, \frac{f^{1,k-m+1,k-m+2}}{f^{1,k-m+2}}\bigg\rangle;
\]
\item[(3)] for $2m-2 \leq k\leq n$, the ideal is
\[
I_k = \bigg\langle \frac{f^{\mu,\mu_k,\nu_k}}{f^{\mu_k,\nu_k}} \: \bigg\vert \: \mu\in \{1, \dots, \mu_k-1\}\cup\{\nu_k-1\}\bigg\rangle.
\]
\end{enumerate}
\end{proposition}

\begin{proof}
To prove that the $S$-modules $F^k$ define a filtration, we need only show that $\gbf{\alpha}_{k} \in F^k$ for all $1\leq k \leq m$. For this, relation \eqref{eqn:relations1} gives
\begin{equation}
\label{eqn:relationsgcd}
D^k - \gcd(D^k,D^{k+1}) = D^k_{k+1} - \gcd(D^k_{k+1},D^{k+1}_k),
\end{equation}
and hence
\[
\frac{f^{k,k+1}}{f^{k+1}} = \frac{\lcm(f^k,f^{k+1})}{f^{k+1}}=\frac{f^k}{\gcd(f^k,f^{k+1})}=\frac{f^k_{k+1}}{\gcd(f^k_{k+1},f^{k+1}_k)}.
\]
Similarly, we have $f^{k,k+1}/f^{k}= f^{k+1}_{k}/\gcd(f^k_{k+1},f^{k+1}_k)$. Therefore
\begin{equation}
\label{eqn:alpha-beta}
\gbf{\alpha}_{k}=\gcd(f^k_{k+1},f^{k+1}_k)\left(\frac{f^{k, k+1}}{f^{k+1}}\mathbf{e}_{k+1}-\frac{f^{k,k+1}}{f^{k}}\mathbf{e}_{k}\right) = \gcd(f^k_{k+1},f^{k+1}_k)\gbf{\beta}_{k}
\end{equation}
for $1 \leq k \leq m$ as required. To prove part (1), we first note that
\[
\frac{F^k}{F^{k-1}}\cong\frac{\langle\gbf{\beta}_{k}\rangle/ ( \langle \gbf{\beta}_1, \dots,\gbf{\beta}_{k-1},\gbf{\alpha}_{k+1}, \dots,\gbf{\alpha}_m\rangle \cap \langle\gbf{\beta}_{k}\rangle)}{\langle\gbf{\alpha}_{k}\rangle/ ( \langle \gbf{\beta}_1, \dots,\gbf{\beta}_{k-1},\gbf{\alpha}_{k+1}, \dots,\gbf{\alpha}_m\rangle \cap \langle\gbf{\alpha}_{k}\rangle )}.
\]
In order to compute this quotient, it suffices, in view of \eqref{eqn:alpha-beta} and the remarks at the beginning of the proof of Proposition~\ref{prop:filtration1}, to determine a set of generators for the module of syzygies on $ \gbf{\beta}_1, \dots,\gbf{\beta}_{k},\gbf{\alpha}_{k+1}, \dots,\gbf{\alpha}_m$ for $1\leq k \leq m$. Proceeding exactly as in the proof of Lemma~\ref{lem:cyclic}, we find that this module is cyclic with generator
\begin{equation}
\label{eqn:sigma0alpha}
\gbf{\sigma}_0:=-\frac{\lcm(f^{1,\dots,m},g^{k+1,k+2},\dots,g^{m,1} )}{f^{1,m}}\gbf{\varepsilon}_m+\sum_{j=1}^{m-1} \frac{\lcm(f^{1,\dots,m},g^{k+1,k+2},\dots,g^{m,1} )}{f^{j,j+1}}\gbf{\varepsilon}_j,
\end{equation}
where we have set $g^{i,i+1}:=\gcd(f^i_{i+1},f^{i+1}_i)$ for $k+1\leq i \leq m$. Ignoring for now the $\Pic(X)$-grading, we deduce from this that
\[
\frac{\langle\gbf{\beta}_{k}\rangle}{ \langle \gbf{\beta}_1, \dots,\gbf{\beta}_{k-1},\gbf{\alpha}_{k+1}, \dots,\gbf{\alpha}_m\rangle \cap \langle\gbf{\beta}_{k}\rangle}\cong \frac{S}{\langle \lcm(f^{1,\dots,m},g^{k+1,k+2},\dots,g^{m,1})/f^{k,k+1}\rangle}.
\]
and therefore, by virtue of \eqref{eqn:alpha-beta},
\[
\frac{F^{k}}{F^{k-1}}\cong \frac{S}{\langle \gcd(f^{k}_{k+1},f^{k+1}_k), \lcm(f^{1,\dots,m},g^{k+1,k+2},\dots,g^{m,1})/f^{k,k+1}\rangle}
\]
which gives the ideal $I_k$ in part (1). For parts (2) and (3), Proposition~\ref{prop:filtration1}(iii) and (iv) respectively determine the ideals $I_k$ for which $F^k/F^{k-1}$ is isomorphic to $S/I_k$ as ungraded rings.

It remains to establish the isomorphism as $\Cl(X)$-graded rings. In light of the above and isomorphism \eqref{eqn:betakquotient}, it suffices to show that the degree of $\gbf{\beta}_k$ is $L_{\mu_k}\otimes L_{\nu_k}\otimes L^{-1}(\gcd(D^{\mu_k},D^{\nu_k}))$ for $1\leq k \leq n$. For each $1 \leq k \leq n$, multiplication by the monomials $f^{\mu_k}$ and $f^{\nu_k}$ define $\Pic(X)$-graded maps $S \to S(L\otimes L_{\mu_k}^{-1})$ and $S \to S(L\otimes L_{\nu_k}^{-1})$ respectively. Tensoring each map with $S(L_{\mu_k}\otimes L_{\nu_k}\otimes L^{-1})$ yields $\Pic(X)$-graded maps $S(L_{\mu_k}\otimes L_{\nu_k}\otimes L^{-1}) \to S(L_{\nu_k})$ and $S(L_{\mu_k}\otimes L_{\nu_k}\otimes L^{-1}) \to S(L_{\mu_k})$ which, in turn, can be combined to form a $\Pic(X)$-graded map
\[
S(L_{\mu_k}\otimes L_{\nu_k}\otimes L^{-1}) \longrightarrow \bigoplus_{j=1}^m S(L_j),
\]
whose image in $\bigoplus_{j=1}^m S(L_j)$ is generated by the element $f^{\mu_k} \mathbf{e}_{\nu_k}-f^{\nu_k}\mathbf{e}_{\mu_k}$. Twisting further by $S(\mathscr{O}_X(\gcd(D^{\mu_k},D^{\nu_k})))$ determines a $\Pic(X)$-graded map
\[
S(L_{\mu_k}\otimes L_{\nu_k}\otimes L^{-1}(\gcd(D^{\mu_k},D^{\nu_k}))) \longrightarrow \bigoplus_{j=1}^m S(L_j)
\]
whose image is generated by the element
\begin{equation}
\label{eqn:nearlybeta}
\frac{f^{\mu_k}}{\gcd(f^{\mu_k},f^{\nu_k})} \mathbf{e}_{\nu_k}-\frac{f^{\nu_k}}{\gcd(f^{\mu_k},f^{\nu_k})}\mathbf{e}_{\mu_k}.
\end{equation}
To prove the claim it remains to show that \eqref{eqn:nearlybeta} coincides with $\gbf{\beta}_k$, but this is immediate since $f^{\mu_k}/\gcd(f^{\mu_k},f^{\nu_k})= \lcm(f^{\mu_k},f^{\nu_k})/f^{\nu_k}$ and $f^{\nu_k}/\gcd(f^{\mu_k},f^{\nu_k}) = \lcm(f^{\mu_k},f^{\nu_k})/f^{\mu_k}$.
\end{proof}

For $1 \leq k \leq n$, each of the generators of $I_k$ listed in Proposition~\ref{prop:filtration2} is a monomial in the Cox ring $S$ of $X$, so its divisor of zeros is an effective torus-invariant Weil divisor in $X$. Notice that while $f^j$, $f^j_{j+1}$, $f^{j+1}_j$, $f_{j,j+1}$ define torus-invariant Cartier divisors $D^j$, $D^j_{j+1}$, $D^{j+1}_{j}$, $D_{j,j+1}$ in $X$, the generators of the ideals $I_k$ are Weil divisors in general.

\begin{definition}
\label{def:Zk}
For each $1 \leq k \leq n$, define a subscheme $Z_k \subset X$ to be the scheme-theoretic intersection of a set of effective Weil divisors depending on $k$ as follows:
\begin{enumerate}
\item[\one] for $1\leq k\leq m$, define $Z_k$ to be the scheme-theoretic intersection of $\gcd(D_{k+1}^k,D^{k+1}_k)$ and the divisor $\lcm\big(D^1,\dots,D^m,\gcd(D_{k+2}^{k+1},D^{k+2}_{k+1}),\dots,\gcd(D_{1}^{m},D^{1}_{m})\big)-\lcm(D^k,D^{k+1})$;
\item[\two] for $m+1 \leq k \leq 2m-3$, define $Z_k$ to be the scheme-theoretic intersection of the divisors $\lcm(D^1,D^{\nu_k},D^{\nu_k+1},\dots,D^m)-\lcm(D^{1},D^{\nu_k})$ and $\lcm(D^1,D^{\nu_k-1},D^{\nu_k})-\lcm(D^{1},D^{\nu_k})$;
\item[\three] for $2m-2\leq k \leq n$, define $Z_k$ to be the scheme-theoretic intersection of the divisors $\lcm(D^{\mu},D^{\mu_k},D^{\nu_k})-\lcm(D^{\mu_k},D^{\nu_k})$ for $\mu \in \{1, \dots, \mu_k-1\}\cup\{\nu_k-1\}$.
\end{enumerate}
\end{definition}

 The subschemes $Z_k\subset X$ are torus-invariant, though some (possibly all) may be empty, see Example~\ref{exa:hex} for an explicit calculation. These subschemes enable us to formulate and prove the main result of this paper (this is Theorem~\ref{thm:mainintro} from the introduction).

\begin{theorem}
\label{thm:main}
Let $X$ be a normal toric variety and let $T^\bullet$ be the complex from \eqref{eqn:Tbullet}, with differentials determined by the Cartier divisors shown in \eqref{eqn:diagram}. Then:
\begin{enumerate}
\item[(1)]$H^0(T^{\bullet}) \cong \mathscr{O}_Z \otimes L$ where $Z$ is the scheme-theoretic intersection of $D^1,\dots,D^m;$
\item[(2)]$H^{-1}(T^{\bullet})$ has an $n$-step filtration 
\[
\im(d^2)=F^0 \subseteq F^1\subseteq \cdots \subseteq F^{n-1}\subseteq F^n=\ker(d^1)
\]
 where, for $1\leq k \leq n$ and for the permutation $\tau_k=(\mu_k,\nu_k)$, we have
\begin{equation}
\label{eqn:sheafquotient1}
F^k/F^{k-1}\cong \mathscr{O}_{Z_k} \otimes L_{\mu_k}\otimes
L_{\nu_k}\otimes L^{-1}(\gcd(D^{\mu_k},D^{\nu_k}));
\end{equation}
\item[(3)]$H^{-2}(T^{\bullet}) \cong \mathscr{O}_D \otimes L(D)$ where $D=\gcd(D_{1,2},D_{2,3},\dots,D_{m,1});$
\item[(4)]$H^{-3}(T^{\bullet})\cong 0$.
\end{enumerate}
\end{theorem}
\begin{proof}
As described at the beginning of this section, the complex $T^\bullet$ arises from a diagram \eqref{eqn:diagrammaintext} of invertible sheaves on $X$ in which the relations \eqref{eqn:relations1} and \eqref{eqn:relations2} hold, and every such diagram determines a complex of $\Cl(X)$-graded $S$-modules of the form \eqref{eqn:complexSmods}, where one can reproduce the original complex $T^\bullet$ by applying the exact functor \eqref{eq:Cox}. In particular, one can calculate the cohomology sheaves of $T^\bullet$ by computing the cohomology modules of \eqref{eqn:complexSmods} and applying the Cox functor. The statement of part (2) then follows from Proposition~\ref{prop:filtration2} and Definition~\ref{def:Zk}. 

For part (1), note that $H^0(T^\bullet)$ is the cokernel of $\bigoplus_i \mathscr{O}
_X(-D^i)\otimes L\hookrightarrow \mathscr{O}
_X\otimes L$, namely the sheaf $\mathscr{O}
_Z\otimes L$ where $Z$ is the scheme-theoretic intersection of $D^1,\dots,D^m$. For part (4), every nonzero map between invertible sheaves is injective, so $H^{-3}(T^\bullet)\cong 0$. It remains to prove part (3). The proof of the analogous statement from \cite[Lemma~3.1]{CL09} does not immediately extend to our setting, as was the case with parts (1) and (4) above, but we can nevertheless adapt the argument as follows. We claim first that if the greatest common divisor $D$ is zero then $H^{-2}(T^\bullet)\cong 0$. We need only show that  complex \eqref{eqn:complexSmods} has no cohomology in degree $-2$. Indeed, suppose $\gbf{\eta}= \sum_{j=1}^m u_j \mathbf{e}_{j,j+1}$ lies in the kernel of $\varphi^2$, so
\[
0=\varphi^2(\gbf{\eta})=\sum_{j=1}^m u_j( f^j_{j+1}\mathbf{e}_{j+1}- f^{j+1}_j \mathbf{e}_{j}).
\]
This translates into the following set of equations: 
\[
u_{j-1} f^{j-1}_{j}=u_{j}f^{j+1}_{j}\qquad 1 \leq j \leq m.
\] 
By relation \eqref{eqn:relations2} we have $f^{j-1}_j f_{j-1,j}=f^{j+1}_j f_{j,j+1}$ for $1 \leq j \leq m$. Consequently, we find that
\begin{equation}
\label{eqn:relationH2}
u_{j-1}f_{j,j+1}=u_{j}f_{j-1,j}, \qquad 1 \leq j \leq m.
\end{equation}
We claim that $f_{j,j+1}$ divides $u_j$ for all $1 \leq j \leq m$. It suffices to prove that $f_{1,2}$ divides $u_1$ by virtue of \eqref{eqn:relationH2}. Let $x_i$ be a prime factor of $f_{1,2}$ with multiplicity $p$. Since by assumption $\gcd(f_{1,2},f_{2,3},\dots,f_{m,1})=1$, it follows that $x_i^p$ does not divide $f_{\nu,\nu+1}$ for some $\nu \neq 1$. Appealing to \eqref{eqn:relationH2} once again, we find that $u_1f_{\nu,\nu+1}=u_{\nu}f_{1,2}$, and thus $x_{i}^p$ divides $u_1 f_{\nu,\nu+1}$. Since $S$ is a unique factorisation domain, this means that $x_i^p$ divides $u_1$, which in turn implies that $f_{1,2}$ divides $u_1$. If we now set $u:=u_1/f_{1,2}$, then equations \eqref{eqn:relationH2} give
\[
u=\frac{u_1}{f_{1,2}}=\frac{u_2}{f_{2,3}}=\cdots=\frac{u_m}{f_{m,1}},
\]
from which it follows that $\gbf{\eta}=u \sum_{j=1}^m f_{j,j+1}\mathbf{e}_{j,j+1}$. Thus, $\gbf{\eta}$ lies in the image of $\varphi^3$, so the complex \eqref{eqn:complexSmods} has no cohomology in degree $-2$ as required. 

To complete the proof of part (3), suppose $D\neq 0$. We can factor $d^3\colon T^{-3}\to T^{-2}$ as a map $L \to L(D)$ followed by a map with no common divisors. By the above argument, the image of $L(D)$ under this map equals the kernel of $d^2\colon T^{-2} \to T^{-1}$. Therefore $H^{-2}(T^{\bullet})$ can be identified with the cokernel of $L \to L(D)$, which is $\mathscr{O}
_{D} \otimes L(D)$.
\end{proof}

\begin{remark}
\label{rem:CautisLogvinenko}
For $m=3$, Theorem~\ref{thm:mainintro} agrees with the statement of the main technical result from Cautis--Logvinenko~\cite[Lemma~3.1]{CL09} (recall from the discussion surrounding Example~\ref{ex:counterexample} above that the assumptions from \emph{loc.~cit.}, namely that $X$ is an arbitrary smooth separated scheme, should be replaced by the assumptions of Theorem~\ref{thm:mainintro}). Parts (1), (3), (4) of Theorem~\ref{thm:mainintro} clearly generalise the analogues from \cite[Lemma~3.1]{CL09}. As for $H^{-1}(T^\bullet)$, we have $m=3$ and hence $n=3$, so Theorem~\ref{thm:mainintro}(2) gives a $3$-step filtration 
\[
\im(d^2)=F^0\subseteq F^1\subseteq F^2\subseteq F^3=\ker(d^1),
\]
and we claim that the successive quotients agree with those of  {\it loc.~cit.}. To justify this we first compute $F^2/F^{1}$. Since $\tau_2 = (2,3)$, Theorem~\ref{thm:mainintro}(2) shows that 
\[
F^2/F^1\cong \mathscr{O}
_{Z_2}\otimes L_{2}\otimes L_{3}\otimes L^{-1}\big(\gcd(D^{2},D^{3})\big),
\] 
where $Z_2$ is the intersection of $\gcd(D^2_{3},D^{3}_2)$ and $\lcm(D^1, D^2, D^3, \gcd(D^3_1, D^1_3))-\lcm(D^2,D^3)$. A direct computation shows that the relation defined by the generator $\gbf{\sigma}_0$ from \eqref{eqn:sigma0alpha} is  
\[
\frac{f^3_1}{\gcd(f^3_1,\widetilde{f}^2_1)}\gbf{\beta}_1 + \frac{\widetilde{f}^1_2f^3_1}{\gcd(f^3_1,\widetilde{f}^2_1)\widetilde{f}^3_2}\gbf{\beta}_2 - \frac{\widetilde{f}^2_1}{\gcd(f^3_1,\widetilde{f}^2_1)}\gbf{\alpha}_3= 0,
\]
where $\widetilde{f}^i_j = f^i_j/\gcd(f^i_j, f^j_i)$. Since $k=2$, the coefficient of $\gbf{\beta}_2$ coincides with the generator $\lcm(f^{1,2,3},\gcd(f^3_1,f^1_3))/f^{2,3}$ of the ideal $I_2$. In particular, the scheme $Z_2$ is the intersection of $\gcd(D^2_{3},D^{3}_2)$ and $\widetilde{D}^1_{2} +  D^3_{1}-\widetilde{D}^3_2 - \gcd(D^3_1, \widetilde{D}^2_1)$, where $\widetilde{D}^i_j$ is the divisor of zeros of the function $\widetilde{f}^i_j$. Permutations are listed as $\tau_1=(1,2), \tau_2=(3,1), \tau_3=(2,3)$ in \cite{CL09}, so after applying permutation $(1,2,3)$ to our indices, we need only invoke the identity 
\[
\widetilde{D}^2_{3} +  D^1_{2}-\widetilde{D}^1_3 - \gcd(D^1_2, \widetilde{D}^3_2) = D^2+\lcm(D^1_2, \widetilde{D}^3_2)-D^3-\widetilde{D}^1_3
\] 
from \cite[p206]{CL09} to see that $Z_2$ is the scheme in the second bullet point of \cite[Lemma~3.1(2)]{CL09}. In order to compare the sheaves, equation \eqref{eqn:relationsgcd} gives $\gcd(D^2, D^3) = D^2 + \gcd(D^2_3, D^3_2) - D^2_3$, and $\mathscr{O}_X
(D^2) = L_2^{-1}\otimes L$ and $\mathscr{O}_X
(-D^2_3) \cong L_3^{-1}\otimes L_{2,3}$ hence
\begin{align*}
L_{2}\otimes L_{3}\otimes L^{-1}\big(\gcd(D^{2},D^{3})\big) & \cong L_{2}\otimes L_{3}\otimes L^{-1}\big(\gcd(D^{2}_3,D^{3}_2)\big) \otimes L_2^{-1}\otimes L\otimes L_3^{-1}\otimes L_{2,3}\\
 & \cong L_{2,3}(\gcd(D^3_{2},D^{2}_{3})\big).
 \end{align*}
Again, applying the permutation $(1,2,3)$ to the indices recovers the sheaf from the second bullet point of \cite[Lemma~3.1(2)]{CL09}, so our description of $F^2/F^1$ agrees with that from \emph{loc.cit.}. A very  similar calculation shows that our unified description of the quotients $F^k/F^{k-1}$ for $k=1,3$ agrees with those of $F^3/F^2$ and $F^1/F^0$ from \cite[Lemma~3.1(2)]{CL09}.
\end{remark}

\begin{example}
\label{exa:hex}
Let $X$ be the smooth toric threefold determined by the fan $\Sigma$ in $\RR^3$ whose one-dimensional cones are generated by the vectors
\[
\text{$v_1=(1,0,1)$, $v_2=(0,1,1)$, $v_3=(-1,1,1)$, $v_4=(-1,0,1)$, $v_6=(1,-1,1)$, $v_7=(0,0,1)$},
\]
where the cones in higher dimension are best illustrated by the height one slice of $\Sigma$ as shown in Figure~\ref{fig:toricfanX}. In particular, the Cox ring of $X$ is $S=\kk[x_1,\dots, x_7]$ and the Cox irrelevant ideal is the monomial ideal $B = (x_3x_4x_5x_6, x_2x_3x_4x_7, x_2x_3x_4x_6, x_1x_5x_6x_7, x_1x_3x_5x_6, x_1x_2x_3x_6)$. 
\begin{figure}[!ht]
 \centering
 \label{fig:toricfanX}
  \psset{unit=1cm}
     \begin{pspicture}(0,-0.2)(2.5,2.4)
         \psline{*-*}(2,1)(1,2)
         \psline{-*}(1,2)(0,2)
  \psline{-*}(0,2)(0,1)
   \psline{-*}(0,1)(1,0)
  \psline{-*}(1,0)(2,0)
  \psline{-}(2,0)(2,1)
   \psline{-}(2,1)(1,1)
   \psline{-}(1,2)(1,1)
   \psline{-}(0,1)(1,1)
   \psline{-}(1,0)(1,1)
\psline{-}(2,1)(1,0)
\psline{-}(1,2)(0,1)
       \psdot(0,0)
       \psdot(1,0)
       \psdot(2,0)
       \psdot(0,1)
       \psdot(1,1)
       \psdot(2,1)
        \psdot(0,2)
       \psdot(1,2)
       \psdot(2,2)
          \rput(2.2,1.2){$v_1$}
  \rput(1.2,2.2){$v_2$}
   \rput(-0.2,2.2){$v_3$}
 \rput(-0.2,0.8){$v_4$}
  \rput(0.8,-0.2){$v_5$}
\rput(2.2,-0.2){$v_6$}
\rput(1.25,1.2){$v_7$}
\end{pspicture}        
  \caption{Height one slice of the fan $\Sigma$ defining the smooth toric threefold $X$}
  \end{figure}
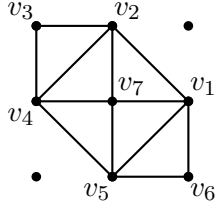
 For $1 \leq \rho \leq 7$, let $E_{\rho}$ denote the divisor in $X$ corresponding to the ray of $\Sigma$ generated by $v_{\rho}$; we use the shorthand $E_{16}=E_{1}+E_6$, $E_{126} = E_1 + E_2 + E_6$ and so on. The group $\Cl(X)$ is the abelian group generated by $E_1,\dots, E_7$ subject to the relations $E_{16} \sim E_{34}$, $E_{23}\sim E_{56}$, and $E_{1234567} \sim 0$ (and since $X$ is smooth, we have that $\Cl(X)$ is isomorphic to the Picard group of $X$).

Set $L:=\mathscr{O}_X$, and consider the diagram of invertible sheaves
 \begin{equation}
 \label{eqn:wheelexample}
  \begin{split}
     \centering    
         \psset{unit=0.45cm}
     \begin{pspicture}(0,-1)(25,16)
\cnodeput*(0,7.5){A}{$L$} 
\cnodeput*(8,15){B}{$L_{1,2}$}
\cnodeput*(8,12){C}{$L_{2,3}$} 
\cnodeput*(8,9){D}{$L_{3,4}$}
\cnodeput*(8,6){S}{$L_{4,5}$}
\cnodeput*(8,3){E}{$L_{5,6}$}
\cnodeput*(8,0){U}{$L_{6,1}$}
\cnodeput*(18,15){F}{$L_{1}$}
\cnodeput*(18,12){G}{$L_{2}$}
\cnodeput*(18,9){H}{$L_{3}$}
\cnodeput*(18,6){T}{$L_{4}$}
\cnodeput*(18,3){I}{$L_{5}$}
\cnodeput*(18,0){V}{$L_{6}$}
\cnodeput*(26,7.5){J}{$L$.}
   \ncline{->}{A}{B}\lput*{:U}(0.5){$\scriptstyle{E_{345}}$}
   \ncline{->}{A}{C}\lput*{:U}(0.5){$\scriptstyle{E_{456}}$}
   \ncline{->}{A}{D}\lput*{:U}(0.5){$\scriptstyle{E_{156}}$}
    \ncline{->}{A}{S}\lput*{:U}(0.5){$\scriptstyle{E_{126}}$}
   \ncline{->}{A}{E}\lput*{:U}(0.5){$\scriptstyle{E_{123}}$}
    \ncline{->}{A}{U}\lput*{:U}(0.5){$\scriptstyle{E_{234}}$}
 \ncline{->}{B}{F}\lput*{:U}(0.3){$\scriptstyle{E_{27}}$}
  \ncline{->}{B}{G}\lput*{:U}(0.3){$\scriptstyle{E_6}$}
  \ncline{->}{C}{G}\lput*{:U}(0.3){$\scriptstyle{E_3}$}
  \ncline{->}{C}{H}\lput*{:U}(0.3){$\scriptstyle{E_{17}}$}
\ncline{->}{D}{H}\lput*{:U}(0.3){$\scriptstyle{E_{47}}$}
\ncline{->}{D}{T}\lput*{:U}(0.3){$\scriptstyle{E_{27}}$}
\ncline{->}{S}{T}\lput*{:U}(0.6){$\scriptstyle{E_{57}}$}
  \ncline{->}{S}{I}\lput*{:U}(0.6){$\scriptstyle{E_3}$}
  \ncline{->}{E}{I}\lput*{:U}(0.6){$\scriptstyle{E_6}$}
  \ncline{->}{E}{V}\lput*{:U}(0.6){$\scriptstyle{E_{47}}$}
\ncline{->}{U}{V}\lput*{:U}(0.6){$\scriptstyle{E_{17}}$}
\nccurve[angleA=-40,angleB=140]{->}{U}{F}\lput*{:U}(0.744){$\scriptscriptstyle{E_{57}}$}
 \ncline{->}{F}{J}\lput*{:U}(0.5){$\scriptstyle{E_{16}}$}
   \ncline{->}{G}{J}\lput*{:U}(0.5){$\scriptstyle{E_{127}}$}
   \ncline{->}{H}{J}\lput*{:U}(0.5){$\scriptstyle{E_{23}}$}
   \ncline{->}{T}{J}\lput*{:U}(0.5){$\scriptstyle{E_{34}}$}
   \ncline{->}{I}{J}\lput*{:U}(0.5){$\scriptstyle{E_{457}}$}
   \ncline{->}{V}{J}\lput*{:U}(0.5){$\scriptstyle{E_{56}}$}
       \end{pspicture}
  \end{split}
  \end{equation}
 where $L_4\cong L_1 = \mathscr{O}
_X(-E_{16})$, $L_5\cong L_2 = \mathscr{O}
_X(-E_{127})$, $L_6\cong L_3 = \mathscr{O}
_X(-E_{23})$, and similarly, where $L_{5,6}\cong L_{3,4}\cong L_{1,2}=\mathscr{O}
_X(E_{345})$, $L_{6,1}\cong L_{4.5}\cong L_{2,3}=\mathscr{O}
_X(E_{456})$. Let $T^{\bullet}$ be the total complex of diagram \eqref{eqn:wheelexample}. With the notation above, the generators $\gbf{\beta}_1,\dots,\gbf{\beta}_{15}$ of $\ker(d^1)$ are
  \begin{equation*}
\begin{array}{lll}
\gbf{\beta}_1= -x_2 x_7 \mathbf{e}_1 + x_6 \mathbf{e}_2, & \gbf{\beta}_6= x_5 \mathbf{e}_1 - x_1 \mathbf{e}_6, & \gbf{\beta}_{11}= -x_4 x_5 \mathbf{e}_2 + x_1 x_2 \mathbf{e}_5,\\
\gbf{\beta}_2= -x_3 \mathbf{e}_2 + x_1 x_7 \mathbf{e}_3, & \gbf{\beta}_7= -x_2 x_3 \mathbf{e}_1 + x_1 x_6 \mathbf{e}_3, & \gbf{\beta}_{12}= -x_5 x_6 \mathbf{e}_2 + x_1 x_2 x_7 \mathbf{e}_6, \\
\gbf{\beta}_3= -x_4 \mathbf{e}_3 + x_2 \mathbf{e}_4, & \gbf{\beta}_8= -x_3 x_4 \mathbf{e}_1 + x_1 x_6 \mathbf{e}_4, & \gbf{\beta}_{13}= -x_4 x_5 x_7 \mathbf{e}_3 + x_2 x_3 \mathbf{e}_5,  \\
\gbf{\beta}_4= -x_5 x_7 \mathbf{e}_4 + x_3 \mathbf{e}_5, & \gbf{\beta}_9= -x_4 x_5 x_7 \mathbf{e}_1 + x_1 x_6 \mathbf{e}_5, & \gbf{\beta}_{14}= -x_5 x_6 \mathbf{e}_3 + x_2 x_3\mathbf{e}_6,\\
\gbf{\beta}_5= -x_6 \mathbf{e}_5 + x_4 x_7 \mathbf{e}_6, & \gbf{\beta}_{10}= -x_3 x_4 \mathbf{e}_2 + x_1 x_2 x_7 \mathbf{e}_4, & \gbf{\beta}_{15}= -x_5 x_6 \mathbf{e}_4 + x_3 x_4 \mathbf{e}_6.\\
\end{array}
\end{equation*}

It is easy to see that the relations
\[
\gbf{\beta}_9 =-x_4 x_7 \gbf{\beta}_6 - x_1 \gbf{\beta}_5,\;\; \gbf{\beta}_{10}=x_4 \gbf{\beta}_2 + x_1 x_7 \gbf{\beta}_3,\;\;\gbf{\beta}_{12}=-x_5\gbf{\beta}_1 - x_2 x_7 \gbf{\beta}_6, \;\; \gbf{\beta}_{13}=x_5 x_7 \gbf{\beta}_3 + x_2 \gbf{\beta}_4
\]
hold, so the successive quotients $F^k/F^{k-1}$ vanish for $k=9,10,12,13$.
 In addition, the generators $\gbf{\alpha}_1,\dots,\gbf{\alpha}_6$ of $\im(d^2)$ satisfy $\gbf{\alpha}_1=\gbf{\beta}_1$, $\gbf{\alpha}_2=\gbf{\beta}_2$, $\gbf{\alpha}_3=x_7 \gbf{\beta}_3$, $\gbf{\alpha}_4=\gbf{\beta}_4$, $\gbf{\alpha}_5=\gbf{\beta}_5$ and $\gbf{\alpha}_6=x_7 \gbf{\beta}_6$, so $F^{k}/F^{k-1}$ also vanishes for $k=1, 2, 4, 5$. 
 
 We now analyse three nonvanishing quotients $F^k/F^{k-1}$ to illustrate part (2) of Theorem~\ref{thm:mainintro}. First consider $k=3$. The transposition $\tau_3=(3,4)$ determines $\gcd(D^3,D^4) = E_3$, so
\[
 F^3/F^{2} \cong \mathscr{O}
_{Z_3}\otimes L_{3}\otimes L_{4}\otimes L^{-1}(E_3)
\]
where, according to Definition~\ref{def:Zk}\one, $Z_3$ is the scheme-theoretic intersection of the effective torus-invariant divisors $\gcd(D^3_{4},D^{4}_3) = E_7$ and
\[
\lcm(D^1, D^2, D^3, D^4, D^5, D^6,\gcd(D^4_5,D^5_4),\gcd(D^5_6,D^6_5),\gcd(D^6_1,D^1_6))-\lcm(D^3,D^4) = E_{1567}.
\]
In particular, $\supp(\mathscr{O}
_{Z_3}) = E_7$.  Now consider the case $k=7$. The corresponding transposition  
$\tau_7=(1,3)$ determines $\gcd(D^1,D^3) = 0$, so
\[
 F^7/F^{6} \cong \mathscr{O}
_{Z_7}\otimes L_{1}\otimes L_{3}\otimes L^{-1}
\]
where, according to Definition~\ref{def:Zk}\two, $Z_7$ is the scheme-theoretic intersection of the divisors $\lcm(D^1,D^2,D^3)-\lcm(D^1,D^3) = E_7$ and $\lcm(D^1, D^3, D^4, D^5, D^6) - \lcm(D^1, D^3) = E_{457}$, giving $Z_7=E_7\cap E_{457}$ and $\supp(\mathscr{O}
_{Z_7}) = E_7$. Finally, consider the case $k=15$ for which the corresponding transposition $\tau_{15}=(4,6)$ determines $\gcd(D^4,D^6) = 0$, so
\[
 F^{15}/F^{14} \cong \mathscr{O}
_{Z_{15}}\otimes L_{4}\otimes L_{6}\otimes L^{-1}
\]
where, according to Definition~\ref{def:Zk}\three, $Z_{15}$ is the scheme-theoretic intersection of the divisors $\lcm(D^{\mu},D^4, D^6) - \lcm(D^4, D^6)$ for $\mu=1, 2, 3, 5$, giving $Z_{15} = E_1\cap E_{127}\cap E_2\cap E_7$. In particular, the support of $\mathscr{O}_{Z_{15}}$ is the torus-invariant point $E_1\cap E_2\cap E_7$ in $X$. 

As for $H^k(T^\bullet)$ for $k\neq -1$, notice that the scheme theoretic intersection of $D^1,\dots, D^6$ is contained in $D^1\cap D^4 = (E_1+E_6) \cap (E_3+E_4) = \emptyset$, so $H^{0}(T^{\bullet}) \cong 0$ by Theorem~\ref{thm:mainintro}(1). Similarly, $\gcd(D_{1,2},D_{2,3},D_{3,4},D_{4,5},D_{5,6},D_{6,1})=0$ so $H^{-2}(T^{\bullet})\cong 0$ by Theorem~\ref{thm:mainintro}(3). It follows that the complex $T^{\bullet}$ has cohomology concentrated in degree $-1$.
\end{example}

\begin{remark}
\label{rem:hex}
One can carry out much of the above calculation using Macaulay2~\cite{M2} in any given example, though the final description of $F^k/F^{k-1}$ is less user-friendly and geometric than ours. To give the flavour, we reproduce some of the calculations from Example~\ref{exa:hex}, omitting for brevity the information on the degree in the $\Cl(X)$-grading of each $S$-module generator\footnote{Macaulay2 require the $\Cl(X)$-degree information in order to create the chain complex $\texttt{T}$, so for convenience we include the complete M2 commands at the end of the latex source file.}.

\medskip

\begin{verbatim}S = QQ[x_1,x_2,x_3,x_4,x_5,x_6,x_7];
d1 = matrix{{x_1*x_6,x_1*x_2*x_7,x_2*x_3,x_3*x_4,x_4*x_5*x_7,x_5*x_6}}
d2 = matrix{{-x_2*x_7,0,0,0,0,-x_5*x_7},{x_6,x_3,0,0,0,0}, 
             {0,-x_1*x_7,x_4*x_7,0,0,0},{0,0,-x_2*x_7,-x_5*x_7,0,0},  
             {0,0,0,x_3,x_6,0},{0,0,0,0,-x_4*x_7,x_1*x_7}}
d3 = matrix{ {-x_3*x_4*x_5},{x_4*x_5*x_6},{x_1*x_5*x_6},{-x_1*x_2*x_6},
             {x_1*x_2*x_3},{x_2*x_3*x_4}}
T = chainComplex(d1,d2,d3)
\end{verbatim}

The minimal generators $\{\gbf{\beta}_j \mid j\in\{1,\dots, 15\}\setminus \{9,10,12,13\}\}$ can be obtained using 
\begin{verbatim}
ker d1
\end{verbatim}
though Macaulay2 chooses an order on these generators that differs from ours. To obtain the cohomology sheaf $H^{-k}(T^\bullet)$ we compute the $k$th cohomology of $\texttt{T}$ and saturate by the irrelevant ideal. For example,  the commands
\begin{verbatim}
B = ideal(x_3*x_4*x_5*x_6,x_2*x_3*x_4*x_7,x_2*x_3*x_4*x_6,x_1*x_5*x_6*x_7, 
           x_1*x_3*x_5*x_6,x_1*x_2*x_3*x_6 )
H0 = prune HH_0(T)
prune (H0/ saturate(0_S*H0,B)) \end{verbatim}
show that $H^{0}(T^\bullet)\cong 0$. Similarly $H^{-2}(T^\bullet)=0$. As for the filtration on $H^{-1}(T^\bullet)$, we input the submodules $F^k$ by hand and compute the quotients, for example,
\begin{verbatim}
F2=image matrix{{-x_2*x_7,0,0,0,0,-x_5*x_7}, {x_6,x_3,0,0,0,0}, 
           {0,-x_1*x_7,x_4*x_7,0,0,0},{0,0,-x_2*x_7,-x_5*x_7,0,0},
           {0,0,0,x_3,x_6,0},{0,0,0,0,-x_4*x_7,x_1*x_7}}

F3=image matrix{ {-x_2*x_7,0,0,0,0,-x_5*x_7}, {x_6,x_3,0,0,0,0},
           {0,-x_1*x_7,x_4,0,0,0}, {0,0,-x_2,-x_5*x_7,0,0},
           {0,0,0,x_3,x_6,0},{0,0,0,0,-x_4*x_7,x_1*x_7}}
Q3 = F3/F2
prune Q3
\end{verbatim}
\noindent In this case, the output is 
\begin{verbatim}
cokernel | x_7 |
\end{verbatim}
so we reproduce our result that $F^3/F^2$ is supported on the divisor $E_7$.  Similar, input
\begin{verbatim}
F15=image matrix{ 
         {-x_2*x_7,0,0,0,0,x_5,-x_2*x_3,-x_3*x_4,-x_4*x_5*x_7,0,0,0,0,0,0},
         {x_6,-x_3,0,0,0,0,0,0,0,-x_3*x_4,-x_4*x_5,-x_5*x_6,0,0,0},
         {0,x_1*x_7,-x_4,0,0,0,x_1*x_6,0,0,0,0,0,-x_4*x_5*x_7,-x_5*x_6,0}, 
         {0,0,x_2,-x_5*x_7,0,0,0,x_1*x_6,0,x_1*x_2*x_7,0,0,0,0,-x_5*x_6}, 
         {0,0,0,x_3,-x_6,0,0,0,x_1*x_6,0,x_1*x_2,0,x_2*x_3,0,0}, 
         {0,0,0,0,x_4*x_7,-x_1,0,0,0,0,0,x_1*x_2*x_7,0,x_2*x_3,x_3*x_4}}
\end{verbatim}
and $\texttt{F14}$ (simply delete the final column in the above), then compute
\begin{verbatim}
Q15 = F15/F14
prune Q15
\end{verbatim}
\noindent In this case, the output is 
\begin{verbatim}
cokernel | x_7 x_2 x_1 |
\end{verbatim}
This confirms our calculation from Example~\ref{exa:hex} that $F^{15}/F^{14}$ is supported on the torus-invariant point $E_1\cap E_2\cap E_7$.  
\end{remark}


\begin{thebibliography}{1}

\bibitem{BPS98}
Dave Bayer, Irena Peeva, and Bernd Sturmfels.
\newblock Monomial resolutions.
\newblock {\em Math. Res. Lett.}, 5(1-2):31--46, 1998.

\bibitem{BCQ12}
Raf Bocklandt, Alastair Craw, and Alexander Quintero~V\'{e}lez, 2013.
\newblock Geometric {R}eid's recipe for dimer models
\newblock Preprint \texttt{arXiv:1305.0156}.

\bibitem{CL09}
Sabin Cautis and Timothy Logvinenko.
\newblock A derived approach to geometric {M}c{K}ay correspondence in dimension
  three.
\newblock {\em J. Reine Angew. Math.}, 636:193--236, 2009.

\bibitem{CCL12}
Sabin Cautis, Alastair Craw, and Timothy Logvinenko, 2012.
\newblock {\em Derived Reid's recipe for abelian subgroups of} $\SL(3,\CC)$. 
\newblock Preprint \texttt{arXiv:1205.3110}.

\bibitem{Cox95}
David Cox.
\newblock The homogeneous coordinate ring of a toric variety. 
\newblock {\em J. Algebraic Geom.}, 4 (1):17--50, 1995.

\bibitem{Eisenbud95}
David Eisenbud.
\newblock {\em Commutative algebra}, volume 150 of {\em Graduate Texts in
  Mathematics}.
\newblock Springer-Verlag, New York, 1995.
\newblock With a view toward algebraic geometry.

\bibitem{M2}
Daniel Grayson and Michael Stillman.
\newblock Macaulay 2, a software system for research in algebraic geometry.
\newblock Available at www.math.uiuc.edu/Macaulay2/.

\bibitem{KR00}
Martin Kreuzer and Lorenzo Robbiano.
\newblock {\em Computational commutative algebra. 1}.
\newblock Springer-Verlag, Berlin, 2000.

\bibitem{Mustata02}
Mircea Musta{\c{t}}{\u{a}}.
\newblock Vanishing theorems on toric varieties.
\newblock {\em Tohoku Math. J. (2)}, 54 (3):451--470, 2002.
     
\bibitem{Nakamura01}
Iku Nakamura.
\newblock Hilbert schemes of abelian group orbits. 
\newblock {\em J. Algebraic Geom.}, 10 (4):757--779, 2001. 

\bibitem{OhsugiHibi}
Hidefumi Ohsugi and Takayuki Hibi.
\newblock Toric ideals generated by quadratic binomials.
\newblock {\em J. Algebra}, 218(2):509--527, 1999.

\bibitem{Villarreal}
Rafael Villarreal.
\newblock {\em Monomial algebras}, volume 238 of {\em Monographs and Textbooks
  in Pure and Applied Mathematics}.
\newblock Marcel Dekker Inc., New York, 2001.

\end{thebibliography}
\end{document}